\documentclass[12pt]{article}
\usepackage[english]{babel}
\usepackage[letterpaper,top=2cm,bottom=2cm,left=3cm,right=3cm,marginparwidth=1.75cm]{geometry}
\usepackage{graphicx}
\usepackage{graphics}
\usepackage{amsfonts}
\usepackage{amscd}
\usepackage{amsthm}
\usepackage{indentfirst}
\usepackage[all]{xy}
\usepackage{amssymb, amsmath, amsthm, amsgen, amstext, amsbsy, amsopn}
\usepackage{epic,eepic}
\usepackage{enumitem} 
\usepackage{color}
\usepackage{ dsfont }
\usepackage[colorlinks=true, allcolors=blue]{hyperref}
\usepackage{verbatim}
\usepackage{subcaption}

\theoremstyle{plain}
\newtheorem{thm}{Theorem}[section]
\newtheorem{lemma}[thm]{Lemma}
\newtheorem{prop}[thm]{Proposition}
\newtheorem{claim}[thm]{Claim}

\newtheorem*{thmsymp*}{Theorem \ref{thm:gen_cover}'}
\newtheorem*{theorem*}{Theorem}

\theoremstyle{plain}
\newtheorem{thmintro}{Theorem}
\newtheorem{propintro}[thmintro]{Proposition}
\newtheorem{question}[thmintro]{Question}

\theoremstyle{definition}
\newtheorem{defin}[thm]{Definition}
\newtheorem{rem}[thm]{Remark}
\newtheorem{setup}[thm]{Setup}
\newtheorem{example}[thm]{Example}
\newtheorem{notation}[thm]{Notation}
\newtheorem{conj}[thm]{Conjecture}

\newcommand{\R}{{\mathbb{R}}}

\newcommand{\Q}{{\mathbb{Q}}}

\newcommand{\Z}{{\mathbb{Z}}}

\newcommand{\D}{{\mathbb{D}}}

\renewcommand{\d}{\Delta} %size of perturbations
\newcommand{\bd}{\mathbf{b}}

\newcommand{\cA}{{\mathcal{A}}}

\newcommand{\cJ}{{\mathcal{J}}}

\newcommand{\cL}{{\mathcal{L}}}
\newcommand{\cM}{{\mathcal{M}}}
\newcommand{\cN}{{\mathcal{N}}}
\newcommand{\cP}{{\mathcal{P}}}

\newcommand{\cU}{{\mathcal{U}}}

\def\id{{1\hskip-2.5pt{\rm l}}}

\newcommand{\indCZ}{{\operatorname{CZ}\,}}

\newcommand{\spec}{{\it spec}}

\newcommand{\val}{{\it val}}
\newcommand{\gw}{{w}}

\newcommand{\im}{\operatorname{im}}
\newcommand{\ind}{\operatorname{ind}}
\renewcommand{\hat}{\widehat}
\definecolor{lev}{rgb}{0.773,0.294,0.549}

\title{A local-to-global inequality for spectral invariants and an energy dichotomy for Floer trajectories}

\author{Lev Buhovsky and Shira Tanny}
\begin{document}
\maketitle

\begin{abstract}
We study a local-to-global inequality for spectral invariants  of Hamiltonians whose supports have a ``large enough" tubular neighborhood on semipositive symplectic manifolds. In particular, we present the first examples of such an inequality when the Hamiltonians are not necessarily supported in domains with contact type boundaries, or when the ambient manifold is irrational. This extends a series of previous works studying locality phenomena of spectral invariants \cite{polterovich2014symplectic,seyfaddini2014spectral,ishikawa2015spectral,humiliere2016towards,ganor2020floer,tanny2021max}. 
A main new tool is a lower bound, in the spirit of Sikorav, for the energy of Floer trajectories that cross the tubular neighborhood against the direction of the negative-gradient vector field.  
\end{abstract}

\section{Introduction and results.}\label{sec:results}
Hamiltonian spectral invariants were introduced by Oh and Schwartz \cite{oh2005construction,schwarz2000action} and are a central tool in studying dynamical properties of Hamiltonian flows. On a closed symplectic manifold $(M,\omega)$, these invariants assign to each Hamiltonian $F:M\times S^1\rightarrow \R$ and a non-zero quantum homology class $a\in QH_*(M)$ a real number, $c(F;a)\in \R$. The spectral invariants are defined using Floer homology, which is a global invariant of the manifold. However, it is known that in some cases these invariants admit a local behavior. The first evidence for a local behavior of spectral invariants was due to Humili\`ere, Le Roux and Seyfaddini \cite{humiliere2016towards}. {They considered Hamiltonians supported in certain disconnected open subsets on aspherical symplectic manifolds\footnote{Recall that a symplectic manifold $ (M,\omega) $ is called aspherical if $ \omega|_{\pi_2(M)}=c_1|_{\pi_2(M)} =0$.}. In this setting they showed that the spectral invariant with respect to the fundamental class $[M]\in QH_*(M)$ is determined by the invariants of the restrictions to the connected components.} 
More explicitly, if $F$ is a Hamiltonian supported on a disconnected subset $V:=\bigsqcup V_i$ satisfying certain conditions\footnote{ Humili\`ere, Le Roux and Seyfaddini assume that the connected components $V_i$ are incompressible Liouville domains.}, and $F_i$ is the restriction of $F$ to $V_i$,  Humili\`ere, Le Roux and Seyfaddini showed that
$$
c(F;[M]) = \max_i c(F_i;[M]).
$$
This formula does not hold for a general homology class.  In 
\cite{ganor2020floer}, Ganor and Tanny proved that an inequality holds for all classes, again on aspherical manifolds  and the same assumptions on the supports. Moreover, they showed that in this setting {spectral invariants are determined by the restriction of the Hamiltonian to a neighbourhood of its support and are independent of the ambient manifold}. Such strong local behavior is known to fail on general {non-aspherical symplectic manifolds}. However, there are no known counterexamples on any closed symplectic manifold to the following inequality, which we refer to as the \emph{max-inequality}:
\begin{equation}
    \label{eq:max_ineq}
c(F;a_1*\cdots *a_L) \leq \max_i c(F_i;a_i),
\end{equation}
where $F=F_1+\cdots+F_L$, $F_i$ are disjointly supported,  $a_i\in H_*(M)\subset QH_*(M)$ and $*$ is the quantum product. There is also an extension of this inequality for general quantum homology classes, which is stated in Theorem~\ref{thm:max-ineq} below. The max-inequality can be interpreted as a local-to-global inequality, which is a weaker locality phenomenon than the aforementioned results {on aspherical symplectic manifolds}.\\

In \cite{tanny2021max} the max-inequality (\ref{eq:max_ineq}) is proved for $a_i=[M]\in QH_*(M)$, in various settings. That paper extended a line of previous works on locality of spectral invariants, including \cite{polterovich2014symplectic,seyfaddini2014spectral,ishikawa2015spectral,humiliere2016towards,ganor2020floer}.
The methods of all of the aforementioned works allowed studying locality of spectral invariants only for restricted types of supports and closed symplectic manifolds.
In particular,  
the Hamiltonians are assumed to be supported in domains with \emph{contact-type} boundaries and the ambient manifold is assumed to be \emph{rational}, meaning that $\omega(\pi_2(M))$ is a discrete subgroup of $\R$. These technical assumptions seem to be a side-effect of the tools used or developed in the above works. It is currently unknown  whether the max-inequality holds without any assumptions on the manifold or the support:
\begin{question}
Does there exist a closed symplectic manifold $(M,\omega)$ and disjointly supported Hamiltonians $F_1$ and $F_2$ such that 
$$
c(F_1+F_2;a*b)>\max\big\{c(F_1;a),c(F_2;b) \big\}
$$
for some $a,b\in H_*(M)\subset QH_*(M)$?
\end{question}

In the current paper we prove that the max-inequality holds for all quantum homology classes,
under the assumption that the connected components of the support are ``far enough" from each other. More explicitly, we present a measurement of (disjoint) collar neighborhoods around the connected components of the support, that is inspired by Sikorav's energy bounds for holomorphic curves \cite{sikorav1994some}. We show that whenever these sizes are larger than the spectral invariants of the restrictions $F_i$, then the max-inequality (\ref{eq:max_ineq}) holds. We work on semipositive manifolds (see definition in (\ref{eq:semipositive})) due to the simpler foundations of Floer homology in this case. However, we expect our methods to be applicable on any closed symplectic manifold, using virtual techniques (e.g. \cite{pardon2016algebraic}).  

\subsection{A Sikorav-type energy bound for Floer trajectories.}
To state the results more formally, let  $(M^{2n},\omega)$ be a closed symplectic manifold and let $V$ be an open subset of $M$ with smooth boundary. Fix a tubular neighborhood $N$ of $\partial V$ in $M\setminus V$, namely, $N\subset M\setminus V$ is diffeomorphic to  $[0,1]\times  \partial V\tilde{\rightarrow}N$, with $\{0\}\times\partial V$ mapped to $\partial V\subset N$. Denote by $\hat V:=\overline{V}\cup N$, of which we think as an extension of $V$, see Figure~\ref{fig:setting}.\\
\begin{figure}
    \centering
    \includegraphics{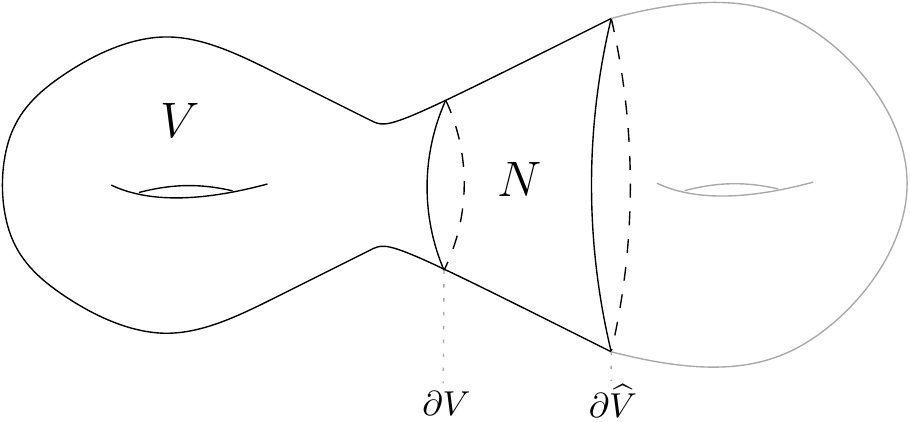}
    \caption{An illustration of the geometric setting. Here $N$ is a smooth tubular neighborhood of the boundary of $V$.}
    \label{fig:setting}
\end{figure}

In \cite{sikorav1994some}, Sikorav showed that the energy of a pseudoholomorphic curve that crosses such a neighborhood $N$ is bounded below  by a constant depending on $N$. In the current paper, we concern Hamiltonian Floer homology and therefore work with Floer trajectories rather than holomorphic curves (see Section~\ref{sec:preliminaries} for preliminaries on Floer homology). Lower bounds for the energy of Floer trajectories that cross certain regions were studied in various works, e.g., \cite{usher2009floer,hein2012conley,groman2021locality} and have been used frequently. These energy lower bounds depend on the Hamiltonian in question and tend to zero when the Hamiltonian does. This phenomenon reflects the fact that Morse flow lines  are a special case of Floer trajectories, and they can travel large distances with small energy, when the gradient of the Hamiltonian is small. On the other hand, when the Hamiltonian is zero, Floer trajectories are pseudoholomorphic curves. Therefore we see a dichotomy between the energies of crossing pseudoholomorphic curves and Morse flow lines, as two particular cases of Floer trajectories. A natural question, which is also motivated by the study of locality in Floer homology, is whether there is a lower bound, that does not tend to zero with the Hamiltonian, for the energies of Floer trajectories  that cross ``against" the negative-gradient direction, as illustrated in Figure~\ref{fig:crossing_traj}. The following theorem gives an affirmative answer to this question.

    \begin{thmintro}%[energy bound]
    \label{thm:main_energy_bound}
        Let $h:{N}\rightarrow\R$ be a smooth function without critical points such that $h|_{\partial V}=0$ and $h|_{\partial \hat V}=1$. Then, there exists $\varepsilon_0$ and a constant $C(N, g_J,h)>0$ such that for any $\varepsilon\in (0,\varepsilon_0)$ and for any homotopy of Hamiltonians $H:M\times S^1\times \R\rightarrow \R$ satisfying $H|_N(x,t,s)=\varepsilon \cdot h(x) {+ \beta(s,t)}$ {for some $\beta:\R\times S^1\rightarrow\R$}, the following holds. Any Floer trajectory $u$  with respect to $(H,J)$ that intersects both $V$ and $M\setminus \hat V$ satisfies 
        \begin{itemize}
            \item {either $x_-\subset M\setminus \hat V$ and $x_+\subset V$}, where $x_\pm:=\lim_{s\rightarrow \pm\infty}u(s,-)$,
            \item or, $E(u)\geq C(N, g_J,h)$.
        \end{itemize}
    \end{thmintro}
    
    \begin{figure}
	\centering
	\begin{subfigure}{.47\textwidth}
		\centering
		\includegraphics[width=0.98\linewidth]{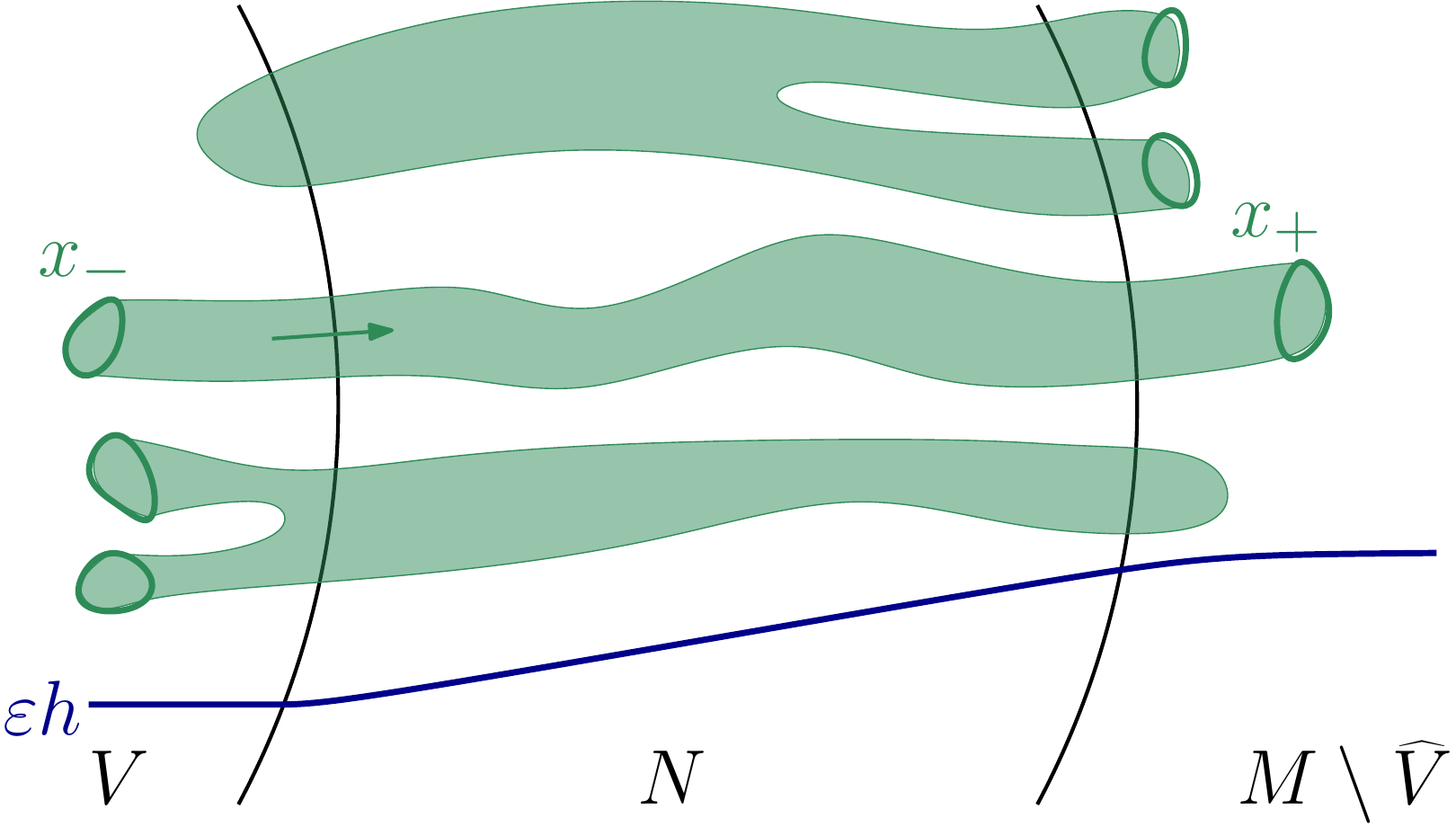}
		\caption{Floer trajectories whose energies are bounded below, analogously to pseudoholomorphic curves.}
		\label{fig:big_enegy}
	\end{subfigure}%
        \hspace{0.5cm}
	\begin{subfigure}{.47\textwidth}
		\centering
		\includegraphics[width=0.98\linewidth]{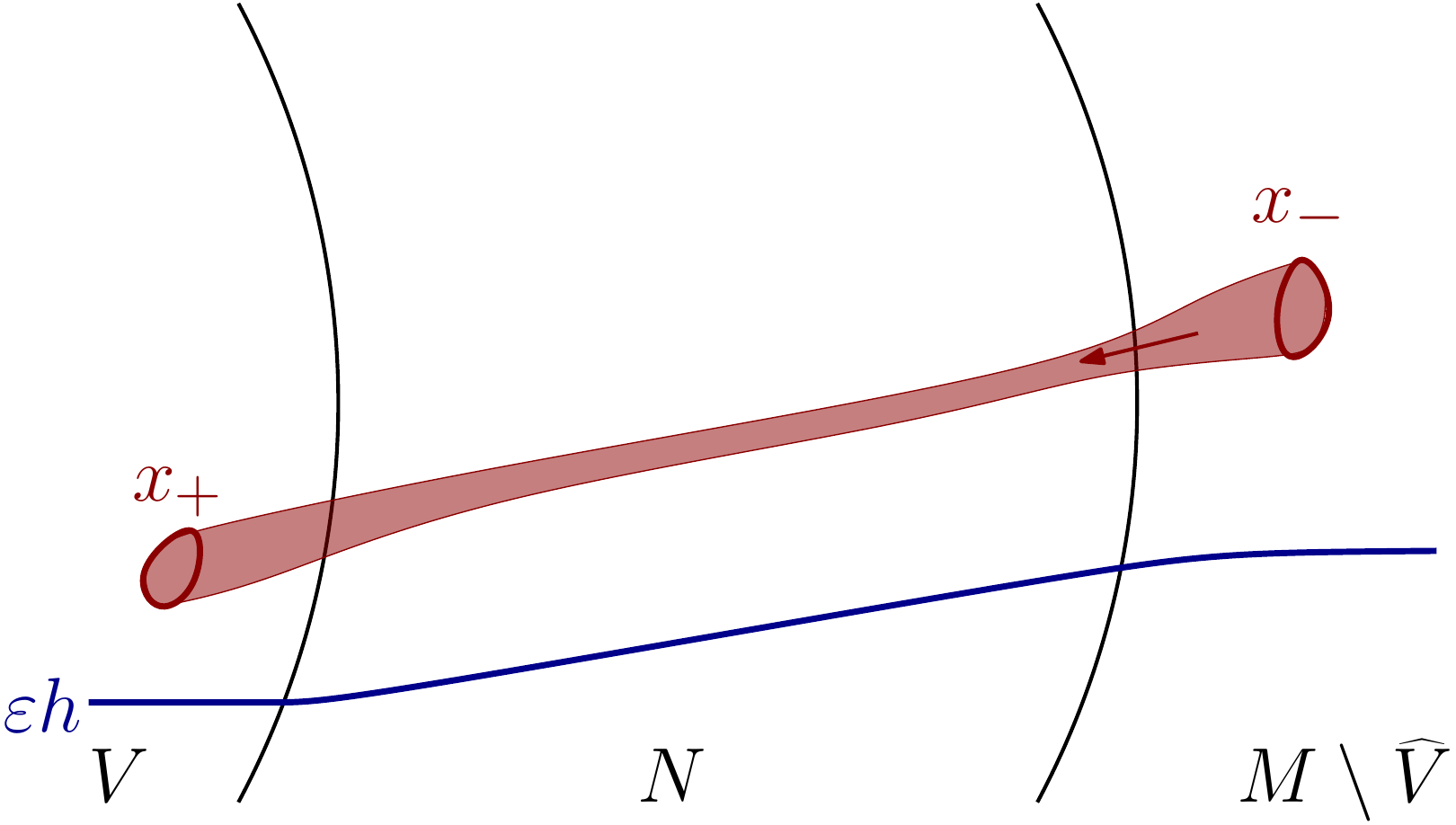}
		\caption{Floer trajectories that go ``down-hill" and might have small energies, like Morse flow-lines.}
		\label{fig:small_energy}
	\end{subfigure}
	\caption{ An illustration of crossing trajectories with and without energy lower bound, as stated in Theorem~\ref{thm:main_energy_bound}.}
	\label{fig:crossing_traj}
\end{figure}
   
    \begin{rem}\label{rem:main_energy_bound}
    \begin{enumerate}[label=(\alph*)]
    \item Theorem~\ref{thm:main_energy_bound} is formulated for homotopies of Hamiltonians, but applies to Hamiltonians as well. Indeed, given a  Hamiltonian $F:M\times S^1\rightarrow\R$, we can identify it with a constant homotopy, namely $H(x,t,s)=F(x,t)$.
    \item In Theorem~\ref{thm:main_energy_bound}, the lower bound depends only on the ``modeling" function $h$ and does not shrink when $H$ tends to a constant on $N$ (that is, when $\varepsilon\rightarrow 0$). 
    \item \label{itm:energy_bnd_perturbation} We actually prove Theorem~\ref{thm:main_energy_bound} for small perturbations of such homotopies as well. Namely, we show that the assertion of Theorem~\ref{thm:main_energy_bound} holds for any homotopy $H$ such that $H|_N(x,t,s)=\varepsilon h(x)+ h'(x,t,s)+\beta(s,t)$ where $h':M\times S^1\times \R\rightarrow \R$ is a $C^\infty$ small homotopy and the support of $\partial_sh'$ is uniformly bounded. Note that we assume that $h'$ is much smaller than $\varepsilon h$.  
    \end{enumerate}
    \end{rem}
    The proof of Theorem~\ref{thm:main_energy_bound} is given in Section~\ref{sec:energy_bound}. It uses the arguments of Sikorav \cite{sikorav1994some} as well as  Hein \cite{hein2012conley}, but requires new arguments as well. 
    Motivated by the Theorem~\ref{thm:main_energy_bound} we define a measurement of the tubular neighborhood $N$, to be the maximal energy lower bound for Floer trajectories that crosses $N$ ``against" the negative-gradient direction. 
    \begin{defin}\label{def:Floer_width} 
    The \emph{Floer width}  of $N$ is:
    \begin{equation}
        w(N):=\sup\left\{C(N, g_J, h)\ |\ J\in\cJ_{\operatorname{reg}},\     h:{N}\rightarrow\R, \ Crit(h)=\emptyset, h|_{\partial V}=0, h|_{\partial \hat V}=1\right\}. 
    \end{equation}
    Here $\cJ_{\operatorname{reg}}$ is a residual subset of the space of all $\omega$-compatible almost complex structures, see Section ~\ref{sec:preliminaries}. 
    \end{defin}
    \begin{example} \label{exa:width_between_balls} Suppose that there exists a symplectic embedding $\psi$ of a ball $B^{2n}(R_+)$ of radius $R_+$ into $M$. Fix $R_-<R_+$ and let $N:=\psi\big(B^{2n}(R_+)\setminus B^{2n}(R_-)\big)$. Then, $w(N) \geq C\cdot (R_+-R_-)^2$, where $C>0$ is a constant depending only on $\dim(M)$.
    \end{example}
    \begin{rem}\label{rem:continuity_of_w}
        {The measurement $w$ is weakly continuous in the following sense. For every $\delta$, there exists a closed tubular sub-neighborhood $N'\subset N\setminus \partial N$ such that $w(N')\geq w(N)-\delta$. The neighborhood $N'$ is obtained as $h^{-1}(c(\delta), 1-c(\delta))$ for some $c(\delta)<1$ and $h$ for which the supremum in Definition~\ref{def:Floer_width} is almost attained.}
    \end{rem}

    \subsection{The max-inequality.}
    Our next main result states that the max-inequality (\ref{eq:max_ineq}) holds for non-negative Hamiltonians supported in domains having disjoint tubular neighborhoods $N_i$, such that the widths $w(N_i)$ (as in Definition~\ref{def:Floer_width})  are large enough, see Figure~\ref{fig:disj_supp_Hams}.
    In this section we assume that the closed symplectic manifold  $(M,\omega)$ is semipositive, namely, for every $A\in \pi_2(M)$,
    \begin{equation}\label{eq:semipositive}
        3-n\leq c_1(A)<0\qquad \Rightarrow \qquad\omega(A)\leq 0.
    \end{equation}
    Semipositive manifolds include Calabi-Yau manifolds, as well as all closed symplectic manifolds of dimension up to 6 (see, e.g., \cite[Section 6.4]{mcduff2012j}).
    See Remark~\ref{rem:foundations} below for a discussion regarding general closed symplectic manifolds.
    \begin{thmintro}\label{thm:max-ineq}
    Let $(M,\omega)$ be a semipositive symplectic manifold and let $F_i:M\times S^1\rightarrow \R$ be non-negative Hamiltonians supported in domains $V_i\subset M$. Suppose that:
    \begin{itemize}
        \item There exist tubular neighborhoods $N_i$ around $V_i$, such that $\hat V_i:=\overline{V_i}\cup N_i$ are pairwise disjoint,
        \item $a_i\in QH_*(M)$ are such that $c(F_i;a_i)-\val(a_i)< w(N_i)$.
    \end{itemize}
    Then,
    \begin{equation}\label{eq:max_ineq_gen_valuation}
        c\Big(\sum_i F_i;*_i a_i\Big)\leq \left(\sum_i \val(a_i)\right) + \max_i \{c(F_i;a_i)-\val(a_i)\}.
    \end{equation}
    In particular, if $a_i\in H_*(M)$, $\val(a_i)=0$ and   $c\Big(\sum_i F_i;*_i a_i\Big)\leq  \max_i c(F_i;a_i)$.
    \end{thmintro}
    \begin{figure}
        \centering
        \includegraphics[width=0.8\textwidth]{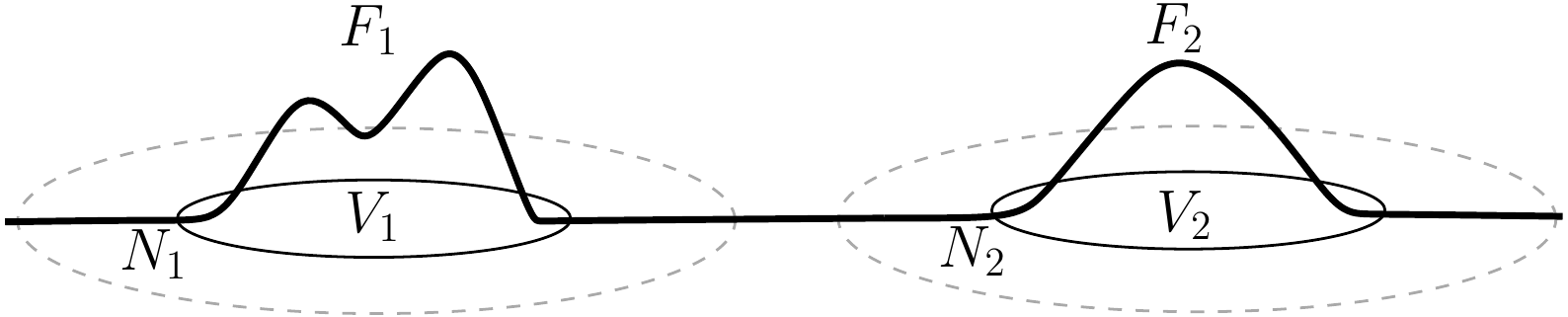}
        \caption{An illustration of Hamiltonians for which we prove the max-inequality.}
        \label{fig:disj_supp_Hams}
    \end{figure}
    \noindent Here $\val:QH_*(M)\rightarrow\R$ is the valuation map whose definition is recalled in Section~\ref{sec:preliminaries}. We remark that inequality (\ref{eq:max_ineq}) does not hold for general quantum homology class without correcting with respect to the valuations, as done in (\ref{eq:max_ineq_gen_valuation}). This can be seen by taking $F=0$ and $a_i = [M]q^A$, since $c(0;a) = \val(a)$.   \\

\begin{example}\label{exa:max_ineq_dim_6}
    Let $(M^{2n},\omega)$ be a closed symplectic manifold of dimension at most $6$, namely, $n\leq 3$. For example, consider  $(M,\omega) = (S^2\times S^2, \omega_a\oplus \omega_b)$ where $\omega_a$ (respectively $\omega_b$) is the area form on $S^2$ of total area $a$ (respectively $b$). Notice that when $a/b\notin \Q$, the symplectic manifold $(M,\omega)$ is irrational.

    There exists a universal constant $\kappa>0$ such that the following holds. Let $\{\psi_i\}_{i=1}^L$ be finite collection of symplectic embedding of balls $B^{2n}(R_i)$ into $M$ with disjoint images $\hat V_i:=\psi_i(B^{2n}(R_i))$. For  every collection of non-negative functions $F_i$ on $M$ that are supported in $V_i:=\psi_i(B^{2n}(\kappa  R_i))$, we have
    \[
    c(F_1+\cdots +F_L; [M])\leq \max_i c(F_i;[M]).
    \]
    This follows from Theorem~\ref{thm:max-ineq}, Example~\ref{exa:width_between_balls} and the energy-capacity inequality (see Section~\ref{sec:preliminaries}).  
\end{example}
    
    As in \cite{seyfaddini2014spectral,tanny2021max}, the proof of inequality (\ref{eq:max_ineq_gen_valuation}) goes through a construction of a Hamiltonian called \emph{spectral killer} which, when added to a Hamiltonian, brings its spectral invariant to the invariant of the function 0. 
    \begin{propintro}
    \label{prop:slow_killer}
    Let $F$ be a non-negative Hamiltonian supported in $V$ and let $a\in QH_*(M)$. If $0<c(F;a)-\val(a)<w(N)$, then there exists a function $K:M\rightarrow\R$, supported on $\hat V$,
    such that \[
    \|K\|_{C^0}=c(F;a)-\val(a) \qquad \text{ and }\qquad c(F+K;a)=\val(a).
    \]
    \end{propintro}

\begin{rem}[Foundations of Floer homology on general closed manifolds] \label{rem:foundations}
    When $(M,\omega)$ is a general closed manifold (i.e., not necessarily semipositive), one needs to consider virtual counts of moduli spaces, such as \cite{pardon2016algebraic}, in order to construct  Floer homology. We expect our results and methods to hold for any type of virtual counts, for which (apart from the well-definedness of Floer homology) the following holds:
    \begin{enumerate}
        \item Given an open subset $U\subset M$ and a moduli space $\cM$ of stable maps such that all of its elements are contained in $U$, the virtual count of $\cM$ depends only on the Floer data (i.e. an almost complex structure and a Hamiltonian / homotopy) restricted to $U$.
        \item The virtual count of an empty moduli space (when there are no stable maps between two given generators) is zero.
        \item The Floer chain complex of a $C^\infty$-small time independent Hamiltonian coincides with its Morse chain complex.
    \end{enumerate}
\end{rem}

\subsection{Application to Polterovich's Poisson bracket invariant}
A central application of the max-inequality (\ref{eq:max_ineq}) concerns the Poisson bracket invariant of covers, which was defined by Polterovich in \cite{polterovich2014symplectic}. This invariant assigns a non-negative number, $pb(\cU)$, to a finite open cover $\cU=\{U_i\}$ of a closed symplectic manifold. The Poisson bracket invariant is known to be strictly positive when the cover consists of displaceable sets. Polterovich conjectured a lower bound for this invariant, which can be interpreted as an uncertainty principle:
\begin{conj}[Polterovich,  \cite{polterovich2014symplectic}]
	Let $(M,\omega)$ be a closed symplectic manifold. There exists a constant $C_M$, depending only on the symplectic manifold $(M,\omega)$, such that for every finite open cover $\cU=\{U_i\}$ of $M$,
	\begin{equation*}
	pb(\cU)\geq \frac{C_M}{e(\cU)},
	\end{equation*}  
	where $e(\cU):=\max_i e(U_i)$ is the maximal displacement energy of a set from $\cU$.
\end{conj}
This conjecture was proved for the case where $M$ is any  surface in \cite{buhovsky2020poisson}, and for surfaces other than the sphere in \cite{payette2018geometry}. In higher dimensions the conjecture is still open. In \cite{entov2006quasi}, Entov, Polterovich and Zapolsky proved a lower bound for $pb$ that decays with the number of sets. Using the arguments of Polterovich \cite{polterovich2014symplectic} one can use the max inequality (whenever it is proved)  to obtain a better lower bound for $pb$ that decays with the {\it degree} of the cover, which is the maximal number of sets intersected by a single set:
\begin{equation*}
d(\cU):=\max_i \#\{j:\bar U_i\cap \bar U_j \neq \emptyset\}.
\end{equation*}
This approach was taken in \cite{polterovich2014symplectic, seyfaddini2014spectral, ishikawa2015spectral}, though these works do not prove a max-inequality, but weaker statements that are sufficient to obtain the lower bound for $pb$. We refer to \cite{tanny2021max} for a more detailed overview of the above works and the relation between the max-inequality and the Poisson bracket invariant. 

{Note that the works mentioned above are all restricted to open covers by sets which are domains with contact-type boundaries.} Moreover, they all require the manifold to be rational, i.e. that $\omega(\pi_2(M))$ is a discrete subgroup of $\R$.  Theorem~\ref{thm:max-ineq} does not require the supports to be contained in domains with contact-type boundaries and therefore significantly extends the scope of open covers  for which  the Poisson bracket invariant can be bounded in terms of the degrees of covers. Moreover, Theorem~\ref{thm:max-ineq} applies to all semipositive manifolds. In particular, our results provide the first examples of irrational symplectic manifolds on which the improved bound for the Poisson bracket invariant is achieved.

\vspace{3pt}

\begin{example}
Let $(M,\omega)$ be a closed symplectic manifold of dimension at most $6$. Let $\cU:=\{U_i\}_i$ be a cover of $M$ by symplectically embedded balls, $U_i:=\psi_i(B^{2n}(r_i))$ such that for each $i$ the embedding $\psi_i$ extends to a symplectic embedding of the ball $B^{2n}(r_i/\kappa)$, where $\kappa>0$ is the constant from Example~\ref{exa:max_ineq_dim_6}. Denoting $\hat U_i:=\psi_i (B^{2n}(r_i/\kappa))$ and $\hat\cU:=\{\hat U_i\}$, we have 
\[
pb(\cU)\geq  \frac{1}{2\cdot d(\hat \cU)^2\cdot \max_i \{\pi r_i^2\} } \geq \frac{1}{2\cdot d(\hat \cU)^2\cdot e(\cU)}.
\]
\end{example}

\subsection*{Acknowledgement.}
   \noindent We thank Leonid Polterovich, Yaniv Ganor, Yael Karshon and Guangbo Xu for useful discussions. L.B. was partially supported by ERC Starting Grant 757585 and ISF grant 2026/17.
   S.T. was supported by a grant from the Institute for Advanced Study School of Mathematics, as well as by  the generosity of Eric and Wendy Schmidt by recommendation of the Schmidt Futures program.

\section{Preliminaries}\label{sec:preliminaries}
Let us review the necessary preliminaries and fix some notations. Note that most of the exposition is restricted to the case when $(M,\omega)$ is semipositive, but we occasionally mention the analogous results for general closed manifolds. 

\subsection{Hamiltonian flows and the action functional.} Given a Hamiltonian $F:M\times S^1\rightarrow\R$, its symplectic gradient is the vector field defined by the equation $\omega(X_F, \cdot) = -dF$ and the flow $\varphi_F^t $ of this vector field is called the Hamiltonian flow of $F$.
The set of 1-periodic orbits of $\varphi_F^t$ is denoted by $\cP(F)$. The Hamiltonian $F$ is called {\it non-degenerate} if the graph of $\varphi_F^1$ is transversal to the diagonal in $M\times M$. Equivalently, $F$ is non-degenerate if every $x\in \cP(F)$ is non-degenerate, that is, if 1 is not an eigenvalue of $d\varphi_F^1(x(0))$ for every $x\in \cP(F)$.

We denote by  $\cL M$ the space of smooth contractible loops in $M$. A capping disk of $x\in\cL M$ is a map $D:\D\rightarrow M$ from the unit disk $\D\subset \R^2$ to $M$, satisfying $D|_{\partial \D}=x$. Two capping disks $D_1, D_2$ of $x$ are equivalent if $[D_1\#(-D_2)]\in \ker \omega\cap \ker c_1$, where $c_1$ is the first Chern class of $M$. 
Throughout the paper we identify $(x,D)$ with the equivalence classes of this relation. We write 
\[
(x_1,D_1)=(x_2,D_2) \quad\text{if}\quad x_1=x_2 \quad\text{and}\quad [D_1\#(-D_2)]\in \ker \omega\cap \ker c_1.
\]
We denote by $\widetilde{\cL M}$ the space of equivalence classes of capped loops, $(x,D)$. 
The action functional corresponding to $F$ is defined on the space $\widetilde{\cL M}$ by
\begin{equation*}
\cA_F(x,D) = \int_0^1 F(x(t),t)\ dt -\int_D \omega.
\end{equation*}
The critical points of the action functional are (equivalence classes of) capped 1-periodic orbits of $\varphi_F^t$ and the set of their values is denoted by $\spec(F)$. \\

\textbf{Almost complex structures.}
Let $\cJ$ be the space of almost complex structures on $M$ that are compatible with $\omega$. Denote by $\cJ_{\operatorname{reg}}\subset\cJ$ the set of \emph{regular} almost complex structures, i.e., all smooth almost complex structures $J\in\cJ$ such that the linearized Cauchy-Riemann operator $D\bar\partial_J(v)$ is
surjective for every simple $J$-holomorphic sphere $v:S^2\rightarrow M$. Note that $\cJ_{\operatorname{reg}}$ is a residual set and thus a generic almost complex structure $J$ on $M$ is regular (see \cite{hofer1995floer}). All of the almost complex structures throughout the paper are assumed to be regular. 

\subsection{The Floer chain complex.}
Let us describe briefly the construction of the Floer chain complex over semipositive manifolds, following \cite{hofer1995floer}.

For a non-degenerate Hamiltonian $F$ and a regular almost complex structure $J$, the Floer chain complex, denoted by $CF_*(F,J)$ or $CF_*(F)$, is spanned over a field  $\Bbbk$ (as the simplest case, one can consider $\Bbbk=\Z_2$) by the  critical points of the action functional, namely, equivalence classes of capped periodic orbits. 
The grading of  $CF_*(F,J)$ is given  by  the Conley-Zehnder (abbreviated to CZ) index, see \cite{robbin1993maslov,gutt2014generalized} for definitions and properties of this index. 

The differential of this chain complex is defined by counting  negative-gradient flow lines of $\cA_F$, with respect to a metric induced by $J$ on $\widetilde{\cL M}$. These negative-gradient flow lines 
are maps $u:\R\times S^1\rightarrow M$ that solve the Floer equation
\begin{equation}\tag{FE}\label{eq:FE}
\partial_su(s,t)+J\circ u(s,t) \cdot\left(\partial_t u(s,t)-X_F\circ u(s,t)\right)=0.
\end{equation}
The {\it energy} of such a solution is defined to be
$
E(u):=\int_{\R\times S^1}\|\partial_s u\|_J^2\ ds\ dt
$, where $\|\cdot\|_J$ is the norm induced by the  inner product associated to $J$, $g_J(\cdot,\cdot):=\omega(\cdot,J\cdot)$. When the Hamiltonian $F$ is non-degenerate, for every solution $u$ with finite energy, there exist $x_\pm \in \cP(F)$ such that $\lim_{s\rightarrow \pm\infty}u(s,t)=x_\pm(t)$. If, in addition, $[D_-\#u\#(- D_+)] =0$,  we say that $u$ {\it connects } $(x_\pm, D_\pm)$. The well known energy identity for such solutions is a consequence of Stokes' theorem:     
\begin{equation}\label{eq:energy_id_Hamiltonian}
E(u):=\int_{\R\times S^1}\|\partial_s u\|^2_{J}\ ds\ dt
=\cA_{F_-}(x_-,D_-)-\cA_{F_+}(x_+,D_+),
\end{equation} 
It immediately follows from the following computation, which uses the Floer equation (\ref{eq:FE}):
\begin{align}\label{eq:energy_identity}
    E(u) &:=\int_{\R\times S^1}\|\partial_s u\|^2_{J}\ ds\ dt
     = \int_{\R\times S^1}\omega(\partial_s u, J\partial_s u)\ ds\ dt \nonumber\\
     &=\int_{\R\times S^1}\omega(\partial_s u, \partial_t u - X_F\circ u )\ ds\ dt \nonumber\\
     &=\int_{\R\times S^1}\omega(\partial_s u, \partial_t u)\ ds\ dt - \int_{\R\times S^1}\omega(\partial_s u,  X_F\circ u )\ ds\ dt \nonumber\\
     &=\int_{\R\times S^1} u^*\omega - \int_{\R\times S^1} u^*dF\wedge dt.
\end{align}
For two capped 1-periodic orbits $(x_\pm,D_\pm)$ of $F$, we denote by $\cM_{(F,J)}((x_-,D_-),(x_+,D_+))$ the set of all solutions $u:\R\times S^1\rightarrow M$ of the Floer equation (\ref{eq:FE}) that connect $(x_\pm, D_\pm)$. Notice that $\R$ acts on this set by translation in the $s$ variable. When $(M,\omega)$ is semipositive and $J\in \cJ_{\operatorname{reg}}$, there exists a residual set of Hamiltonians, called \emph{regular Hamiltonians}, for which the space $\cM_{(F,J)}((x_-,D_-),(x_+,D_+))$ is a smooth manifold whenever $\indCZ(x_-,D_-)-\indCZ(x_+,D_+)\leq 2$. The dimension of this manifold is $\indCZ(x_-,D_-)-\indCZ(x_+,D_+)$. We sometimes abbreviate $\ind(u):=\indCZ(x_-,D_-)-\indCZ(x_+,D_+)$ for $u\in \cM_{(F,J)}((x_-,D_-),(x_+,D_+))$.  The fact that this moduli space is a smooth manifold amounts to the surjectivity of the linearized Floer equation at every $u$ in the space. We therefore adopt the following terminology: we say that a solution $u$ to the Floer equation is \emph{regular} if the linearized Floer equation is surjective at $u$. Note that if all of the elements of $\cM_{(F,J)}(x_-,x_+)$ are regular then it is a smooth manifold as mentioned above. 
Dividing $\cM_{(F,J)}(x_-,x_+)$ by the $\R$ action, we obtain a manifold of dimension $\indCZ(x_-,D_-)-\indCZ(x_+,D_+)-1$. When the difference of CZ indices is 1, we obtain a zero dimensional manifold that is known to be compact, and thus is a finite set. In this case, the Floer differential is defined by a signed count of the points in this set.
\begin{align*}
\partial_{(F,J)}&:CF_*(F)\rightarrow CF_{*-1}(F)\\
\partial_{(F,J)}(\alpha) &:= \sum_{\tiny{\begin{array}{c} 
		(x_\pm,D_\pm)\text{ such that}\\ \indCZ(x_+,D_+)=\indCZ (x_-,D_-)-1
		\end{array}}} \alpha_{(x_-,D_-)}\cdot \#\left(\frac{\cM_{(F,J)}((x_-,D_-),(x_+,D_+))}{\R}\right)\cdot (x_+,D_+),
\end{align*}
where $\alpha=\sum_{(x,D)} \alpha_{(x,D)}\cdot (x,D)\in CF(F,J)$ and $\#$ is a signed $\Z$ count, 
see e.g. \cite{hofer1995floer,mcduff2012j}. 
\begin{rem}
On general closed symplectic manifolds the above moduli spaces  are not necessarily smooth manifolds. Roughly speaking, they are zeros of non-transverse sections of some Banach bundle. There are several works offering solutions to this issue through a replacement of  the count of elements of these moduli spaces by a \emph{virtual count}. This enables to define the Floer homology with coefficients in  $\Q$, see e.g.,  \cite{pardon2016algebraic,filippenko2022polyfold}. Recently, the construction of Floer homology over general closed manifold was upgraded to have $\Z$-coefficients \cite{bai2022arnold,rezchikov2022integral}. 
\end{rem}

\subsection{Floer homology and quantum homology.}
The homology of the complex $(CF_*(F),\partial_{(F,J)})$ is denoted by $HF_*(F,J)$ or $HF_*(F)$. A fundamental result in Floer theory states that Floer homology is isomorphic to  the quantum homology, which, as a vector space, coincides with the singular homology tensored with the Novikov ring. In more detail, let $\Gamma:=\pi_2(M)/(\ker(\omega)\cap\ker(c_1))$, then the Novikov ring is:
\begin{align*}
    \Lambda:= \Big\{ \sum_{A\in \Gamma} c_A q^A\ \big|\ c_A\in \Bbbk,\ \forall C\in\R,\#\{A\in \Gamma: c_A\neq 0, \omega(A)>C\}<\infty \Big\}.
\end{align*}
The Novikov ring has a natural functional, called \emph{valuation}, defined by 
$$
\val:\Lambda\rightarrow \R,\qquad \val\left(\sum_{A\in \Gamma} c_A q^A\right):= \max\{\omega(A):c_A\neq 0\}.
$$
The quantum homology is then given by $QH_*(M) = H_*(M;\Bbbk)\otimes \Lambda$, with the grading of $q^A$ being $2c_1(A)$. The valuation extends to it in the trivial way, $\val:QH_*(M)\rightarrow\R$. The quantum homology carries a product, called the \emph{quantum product} and denoted by $*$, which is a deformation of the classical intersection product on singular homology (see, for example,  \cite[Section 12.1]{polterovich2014function}).

When $F$ is a time independent  $C^2$-small Morse function, its 1-periodic orbits are its critical points, $\cP(F)\cong Crit(F)$, and their cappings correspond to elements in $\pi_2(M)$, or in $\Gamma$ after passing to equivalence classes.  In this case, the Floer complex with respect to $J\in\cJ_{\operatorname{reg}}$  coincides (up to a degree shift) with the Morse complex tensored with the Novikov ring:
\begin{equation*}
\left(CF_*(F), \partial_{(F,J)}^{Floer}\right) = 
\left(CM_{*+n}(F)\otimes \Lambda, \partial_{(F,\left<\cdot, \cdot\right>_J)}^{Morse}\otimes \id_{\Lambda}\right).
\end{equation*} For a proof, see, for example, \cite{hofer1995floer,mcduff2012j} for semipositive manifolds, \cite{audin2014morse} for aspherical manifolds and  \cite[Chapter 10]{pardon2016algebraic} for general closed symplectic manifolds.

\subsubsection{Specialized notations.}
We conclude this section by fixing notations that will be used later on. 
\begin{notation}\label{not:CF_element_contained}
	Let $\alpha=\sum_{(x,D)} \alpha_{(x,D)}\cdot (x,D)$ be an element of $CF_*(F)$.
	
	\begin{itemize}
		\item We say that $(x,D)\in \alpha$ if $\alpha_{(x,D)}\neq 0$.
		\item We denote the maximal action of an orbit from $\alpha$ by $\cA_F(\alpha):=\max\{\cA_F(x,D):\alpha_{(x,D)}\neq 0\}$. 
		
		\item For a subset $X\subset M$, let $C_X(F)\subset CF_*(F)$ be the subspace spanned by the 1-periodic orbits of $F$ that are contained in $X$.	Moreover, let $\pi_X:CF_*(F)\rightarrow C_X(F)$ be the projection onto this subspace. Note that $C_X(F)$ is not necessarily a subcomplex, and $\pi_X$ is not a chain map in general.
	\end{itemize}
\end{notation}

\subsection{Communication between Floer complexes using homotopies.}
Let $H:M\times S^1\times\R\rightarrow\R$ be a homotopy of Hamiltonians. Throughout the paper, we consider only homotopies that are constant outside of a fixed compact set. Namely, there exists $R>0$ such that $\partial_s H|_{|s|>R}=0$, and we denote by $H_\pm(x,t) := \lim_{s\rightarrow\pm \infty}H(x,t,s)$ the ends of the homotopy $H$. 
Given an almost complex structure $J\in \cJ_{\operatorname{reg}}$, we consider the Floer equation (\ref{eq:FE}) with respect to the pair $(H,J)$:
\begin{equation*}
\partial_su(s,t)+J\circ u(s,t) \cdot\left(\partial_t u(s,t)-X_{H_s}\circ u(s,t)\right)=0,
\end{equation*} 
where $H_s(\cdot,\cdot):=H(\cdot,\cdot,s)$. We sometimes refer this equation as ``the $s$-dependent Floer equation", to stress that it is defined with respect to a homotopy of Hamiltonians. For capped 1-periodic orbits $(x_\pm,D_\pm)$, we denote by $\cM_{(H,J)}((x_-,D_-),(x_+,D_+))$ the set of all solutions $u:\R\times S^1\rightarrow M$ of the $s$-dependent Floer equation  that satisfy $\lim_{s\rightarrow \pm\infty}u(s,t)=x_\pm(t)$ and $[D_-\#u\#(-D_+)]=0$. 
The  energy identity for homotopies is:   
\begin{align}\label{eq:energy_identity_homotopies}
    E(u) &:=\int_{\R\times S^1}\|\partial_s u\|^2_{J}\ ds\ dt
     = \int_{\R\times S^1}\omega(\partial_s u, J\partial_s u)\ ds\ dt \nonumber\\
     &=\int_{\R\times S^1}\omega(\partial_s u, \partial_t u - X_{H_s}\circ u )\ ds\ dt \nonumber\\
     &=\int_{\R\times S^1}\omega(\partial_s u, \partial_t u)\ ds\ dt - \int_{\R\times S^1}\omega(\partial_s u,  X_{H_s}\circ u )\ ds\ dt \nonumber\\
     &=\int_{\R\times S^1} u^*\omega -\int_{\R\times S^1} dH(\partial_su)ds\wedge dt\nonumber\\
     &=\int_{\R\times S^1} u^*\omega - \int_{\R\times S^1} d(H\circ u\ dt) + \int_{\R\times S^1} \partial_s H\circ u\  ds\ dt.
\end{align}
From this computation one can see that
\begin{equation}\label{eq:energy_id_homotopies}
	E(u)=\cA_{H_-}(x_-,D_-)-\cA_{H_+}(x_+,D_+) +\int_{\R\times S^1}\partial_s H\circ u\ ds\ dt. 
\end{equation}
As in the case of Hamiltonians, there  exists a residual set of homotopies called regular, for which the spaces $\cM_{(H,J)}((x_-,D_-),(x_+,D_+))$ are smooth compact manifolds whenever $\indCZ(x_-,D_-)-\indCZ(x_+,D_+)\leq 2$. Note that the ends $H_\pm$ of a regular homotopy $H$ are regular Hamiltonians.
\begin{rem}[Achieving regularity.]\label{rem:achiving_regularity}
    Let us review useful known results regarding achieving regularity of moduli spaces via perturbations of Hamiltonians or homotopies. 
    \begin{enumerate}
    \item Given  $J\in \cJ_{\operatorname{reg}}$ and a non-degenerate Hamiltonian $F$, it is sufficient consider perturbations of $F$ that agree with $F$ up to second order on its periodic orbits, see e.g. \cite{audin2014morse} or \cite[Section 9]{ganor2020floer}\footnote{The results of \cite{audin2014morse,ganor2020floer} are stated for aspherical manifolds, i.e. when $\pi_2(M)=0$. However, the transversality proofs are identical in our case, given the fact that on semipositive manifolds and for $J\in \cJ_{\operatorname{reg}}$ one can avoid sphere bubbling, see e.g. \cite{hofer1995floer}.}. In particular, such generic perturbations have the same periodic orbits and action spectrum as $F$.
     \item 
     As proved  in \cite[Section 9]{ganor2020floer}, given a generically chosen $J$ and a homotopy $H$ it is sufficient to consider perturbations $H':M\times S^1\times \R\rightarrow\R$ that are independent of $s$ when $s\notin [-R,R]$  for some fixed $R>0$. Therefore, we consider in this paper only homotopies $H$ such that $\partial_sH$ is supported in $M\times S^1\times [-R,R]$ for some fixed $R>0$.
    \end{enumerate}
\end{rem}

The {\it continuation map} is a degree-preserving chain map between the Floer complexes of the ends, $\Phi:CF_*(H_-)\rightarrow CF_*(H_+)$, defined by 
\begin{equation}\label{eq:continuation_def} 
\Phi(\alpha)=\sum_{\tiny{\begin{array}{c} 
		(x_\pm,D_\pm),\\ \indCZ(x_+,D_+)=\indCZ(x_-,D_-)
		\end{array}}}\alpha_{(x_-,D_-)}\cdot \#\cM_{(H,J)}((x_-,D_-),(x_+,D_+))\cdot (x_+,D_+).
\end{equation} 
The map $\Phi$  induces isomorphism on homology.  When the homotopy $H$ is independent of $s$, i.e. $H=H_-=H_+$, the moduli space $\cM_{(H,J)}((x_-,D_-),(x_+,D_+))$ is zero dimensional and invariant under $\R$-translation, and thus contains only the constant solution:
\[
\cM_{(H,J)}((x_-,D_-),(x_+,D_+))= \begin{cases}
    \emptyset, &\text{ if }(x_-,D_-)\neq (x_+,D_+),\\
    u\equiv x_- , &\text{ if }(x_-,D_-)= (x_+,D_+).
\end{cases}
\]
As a consequence, the continuation map in this case is simply the identity $\id:CF_*(H_-)\rightarrow CF_*(H_-)$. The next lemma is a local version of this phenomenon, allowing for perturbations as well. 
\begin{lemma}\label{lem:almost_const_homotopy}
    Let $U\subset M$ be a open subset and let $H:M\times S^1\times \R\rightarrow\R$ be a homotopy such that:
    \begin{enumerate}
        \item $(H_\pm, J)$ are regular, %the periodic orbits of $H_-$ and $H_+$ in $V$ coincide,\item 
        \item In $U$, $\partial_s H$ vanishes for all $s\in \R$ and $t\in S^1$ (in particular, $H_-|_U=H_+|_U$)\footnote{To simplify the notation, we abbreviate $H|_{U\times S^1\times \R}$ to simply $H|_U$.}.
    \end{enumerate}
    Fix $(x_\pm, D_\pm)$ in $U$ of the same index. Suppose that for any homotopy $H'$ such that $H'|_U$ is $C^\infty$-close to $H|_U$ and $H'_\pm=H_\pm$, every  
    $
    u\in \cM_{(H',J)}((x_-,D_-),(x_+,D_+))
    $
     is contained in $U$. Then \[
     \#\cM_{(H',J)}((x_-,D_-),(x_+,D_+))= \begin{cases}
    0, &\text{ if }(x_-,D_-)\neq (x_+,D_+),\\
    1, &\text{ if }(x_-,D_-)= (x_+,D_+).
\end{cases}
\] In particular, when $U=M$, the  continuation map   $\Phi':CF_*(H_-')\rightarrow CF_*(H_+')$ corresponding to $H'$ is the identity. 
\end{lemma}
\begin{proof}
    An analogous claim for the case where $M$ is aspherical was proved in \cite[Section 9.3.3]{ganor2020floer}. The only difference is that in our case one should account for sphere bubbles whenever taking limits. Ruling out sphere bubbles is possible due to the semipositive condition. For the convenience of the reader we include a sketch of the argument. 
    
    \vspace{3pt}
    
    Let $\{H_\lambda\}_{\lambda\in[0,1]}$ be a path of homotopies starting at $H_0=H$ and ending at $H_1=H'$. Suppose this path is constant near $\lambda=0,1$ and the ends $(H_\lambda)_\pm=H_\pm$ are the same for all $\lambda$. Moreover, we assume that for all $\lambda$, $H_\lambda|_U$ is sufficiently close to $H|_U$. Fix $(x_\pm, D_\pm)$ of the same index and assume that all $u\in \cM_{(H_\lambda, J)}(x_\pm, D_\pm)\subset U$ for all $\lambda$. Since $H$ coincides on $U$ with the regular Hamiltonian (or, constant homotopy) $H_-$, all solutions with respect to $(H,J)$ that are contained in $U$ are regular. In particular, this means that all elements of $\cM_{(H_0, J)}(x_\pm, D_\pm)$ are regular and thus $\cM_{(H_0, J)}(x_\pm, D_\pm)$ is a smooth zero dimensional manifold. The semipositive condition guarantees that it is also compact (a similar argument is sketched below. Alternatively, see \cite[Theorem 3.3]{hofer1995floer}).
    Consider the parametric moduli space \[
    \cM:=\{(\lambda, u): u\in \cM_{(H_\lambda, J)}(x_\pm,D_\pm), \lambda \in[0,1]\}.
    \]
    For a generic such path $\{H_\lambda\}$ (with fixed $H_0, H_1$ that are regular in $U$) the space $\cM$ is a smooth manifold (see \cite[Claim 9.33]{ganor2020floer}). Let us show that when  $H_\lambda|_U$ is close enough to $H|_U$ for all $\lambda$, the manifold $\cM$ is compact. This will enable us to view $\cM$ as a cobordism between the moduli space with respect to $H=H_0$ and $H_1=H'$ and conclude the proof.

    Take any sequence $(\lambda_n, u_n)\in \cM$. After passing to a subsequence, we may assume that $\lambda_n\rightarrow \lambda_\star$ and $u_n$ 
    admits a subsequence converging to a broken trajectory $(v_1,\dots, v_k)$ with respect to $(H_{\lambda_\star},J)$, with $J$-holomorphic sphere bubbles $(w_1,\dots, w_\ell)$. Since $u_n$ are all contained in $U$, their limit is contained in the closure $\overline{U}$ of $U$. Notice that by continuity, $H$ coincides with $H_-$ on $\overline{U}$, and $H_\lambda$ are close to $H$ on $\overline{U}$.  
    Returning to the Floer trajectories in the limit of $u_n$, all of the $v_i$ but one (denote it by $v_{i_0}$) are solutions of the Floer equation with respect to $(H_\pm, J)$, which are regular $s$-independent pairs. In particular, $\ind(v_i)\geq 0$ for all $i\neq i_0$. Finally, $v_{i_0}$ is a solution with respect to $H_{\lambda_\star}$ that is contained in $\overline{U}$. When all homotopies $H_\lambda$ are close enough to $H$ on $\overline{U}$, the solution $v_{i_0}\subset \overline{U}$ cannot have a negative index (otherwise take a sequence of such $H_\lambda|_{\overline{U}} \rightarrow H|_{\overline{U}}$ and by similar arguments obtain a solutions with respect to $H$, that is contained in $\overline{U}$ and has a negative index). Overall we conclude that $\ind(v_i)\geq 0$ for all $i$.
    Recall that
    \begin{align*}
    \indCZ(x_-,D_-) - \indCZ(x_+,D_+) &= \ind(u_n) = \sum_{i=1}^k \ind(v_i)+\sum_{j=1}^\ell 2c_1(w_j).
    \end{align*}
    Since $u_n$ where continuation trajectories, $\indCZ(x_-,D_-) - \indCZ(x_+,D_+)=0$. On semipositive manifolds, $J$-holomorphic spheres cannot have negative Chern class (see \cite[Proposition 2.3]{hofer1995floer}). As we explained above, $\ind(v_i)\geq 0$ for all $i$ as well. Therefore, 
    \[
    \ind(v_i)=0=c_1(w_j) \qquad \text{for all}\qquad i,j.  
    \]
    Recall that for $i\neq i_0$, $v_i$ are solutions with respect to the $s$-independent Hamiltonians $H_\pm$ for which all index zero Floer solutions are constant. Therefore we conclude that the limit of $u_n$ contains a single Floer solution $v_{i_0}$ with respect to $H_{\lambda_\star}$. When all of the $H_\lambda$ are close enough to $H$, we must have $x_-=x_+$ and $v_{i_0}$ is close to the constant solution $u\equiv x_-$ (indeed, otherwise take a sequence of $H_\lambda$ converging to $H$ and by compactness obtain a Floer solution with respect to $(H,J)$ of index 0, which must be a constant solution). As a consequence, the $J$-holomorphic spheres $w_j$ must intersect an arbitrarily small neighborhood of $x_-$. By \cite[Theorem 3.1]{hofer1995floer}, the set if images of all non-constant $J$-holomorphic spheres of Chern class 0 is compact and does not intersect the periodic orbit of a generic Hamiltonian. Hence it will not intersect a small enough open neighborhood of $x_-$. We conclude that $w_j$ are all constant, which implies that, up to passing to a subsequence, the limit of $(\lambda_n,u_n)$ is simply $(\lambda_\star,v_{i_0})\in \cM$.  
    
    As a consequence, $\cM$ is a compact smooth 1-dimensional manifold. Its boundary is simply 
    \[
    \{1\}\times \cM_{(H_1, J)}(x_\pm,D_\pm)\cup -\Big(\{0\}\times\cM_{(H_0, J)}(x_\pm,D_\pm)\Big)
    \]
    where the minus sign accounts for reversing orientation. In particular, 
    \begin{align*}
        \# \cM_{(H_1, J)}(x_\pm,D_\pm)&=\# \cM_{(H_0, J)}(x_\pm,D_\pm) \\
        &= \# \cM_{(H_-, J)}(x_\pm,D_\pm) = \begin{cases}
            0, &\text{ if }(x_-,D_-)\neq (x_+,D_+),\\
            1, &\text{ if }(x_-,D_-)= (x_+,D_+),
        \end{cases}
    \end{align*}
    where the first equality in the second row is due to our assumption that the elements of $\cM_{(H_0, J)}(x_\pm,D_\pm)$ are contained in $U$, where $H_0=H$ coincides with $H_-$.
\end{proof}
\begin{rem}\label{rem:lemma_in_perliminaries}
    The analogous statement for the above lemma on general closed symplectic manifolds should be the following. Suppose that $H$ is as in the statement of Lemma~\ref{lem:almost_const_homotopy}, and that all stable maps connecting $(x_\pm, D_\pm)$ with respect to $(H,J)$ are contained in $U$. Then, the virtual count of $ \cM_{(H, J)}(x_\pm,D_\pm)$ coincides with that of $ \cM_{(H_-, J)}(x_\pm,D_\pm)$ and equals to 1 if $(x_-,D_-)=(x_+,D_+)$  and zero otherwise.
\end{rem}

\subsection{Spectral invariants.}
The Floer complex admits a natural filtration by the action value. Let $CF_*^\lambda(F,J)$ be the sub-complex generated by (equivalence classes of) capped 1-periodic orbits whose action is bounded by $\lambda$ from above. Since the Floer differential is action decreasing, it restricts to the sub-complex $CF_*^\lambda(F,J)$ and the  homology $HF_*^\lambda(F,J)$ is well defined. 
The spectral invariant with respect to a non-zero class $a\in QH_*(M)$  is defined to be the smallest value of $\lambda$ for which the  class $a$ appears in $HF^\lambda_*(F,J)$, namely,
\begin{equation}
c(F;a) := \inf\{\lambda:a\in \im(\iota^\lambda_*)\},
\end{equation}
where $\iota^\lambda_*:HF_*^\lambda(F,J)\rightarrow HF_*(F,J)$ is the map induced by the inclusion $\iota^\lambda:CF_*^\lambda(F,J)\hookrightarrow CF_*(F,J)$.
Spectral invariants have several useful properties, let us state the relevant ones:  
\begin{itemize}
	\item (stability)  For any Hamiltonians $F$ and $G$,
	\begin{eqnarray} 
	\nonumber\int_{0}^{1} \min_{x\in M}(F(x,t)-G(x,t))dt \leq c(F;a) - c(G;a) \leq  \int_{0}^{1} \max_{x\in M}(F(x,t)-G(x,t))dt.
	\end{eqnarray} 
	In particular, $c(-;a):C^\infty(M\times S^1)\rightarrow\R$ is a continuous functional and is extended by continuity to degenerate Hamiltonians. Moreover, this implies that the spectral invariant is monotone: If $G(x,t)\leq F(x,t)$ for all $(x,t)\in M\times S^1$, then $c(G;a)\leq c(F;a)$.
	
	\item (spectrality) $c(F;a)\in\spec(F)$. 
	
	\item (subadditivity) For every Hamiltonians $F$ and $G$, and non-zero classes $a,b\in QH_*(M)$, one has $c(F\#G;a*b)\leq c(F;a)+c(G;b)$, where $F\# G:= F+G\circ(\varphi_F^t)^{-1}$ and $a*b$ is the quantum product of $a$ and $b$. Note that if $F$ and $G$ are disjointly supported then $F\#G = F+G$.
	\item (identity) For every non-zero $a\in QH_*(M)$, $c(0;a) = \val(a)$. 
	\item(energy-capacity inequality) If the support of $F$ is displaceable, its spectral invariants are bounded by the displacement energy of the support, namely, $c(F;a)-\val(a)\leq e(supp(F))$.
	We remind that a subset $X\subset M$ is displaceable if there exists a Hamiltonian $G$ such that $\varphi_G^1(X)\cap X=\emptyset$. In this case, the displacement energy of $X$ is given by 
	\begin{equation}\label{eq:disp_energy_def}
	e(X):=\inf_{G: \varphi_G^1(X)\cap X=\emptyset} \int_0^1 \left(\max_M G(\cdot,t)-\min_M G(\cdot,t)\right)\ dt. 
	\end{equation}
\end{itemize}
For a wider exposition see, for example, \cite{mcduff2012j,polterovich2014function}.
\subsection{Boundary depth.}
{In \cite{usher2011boundary}, Usher defined the {\it boundary depth} of a Hamiltonian $F$ to be the largest action gap between a boundary term in $CF_*(F)$ and its primitive having the smallest action.} 
\begin{defin}\label{def:boundary_depth}
     The boundary depth of $F$ is:
\begin{equation*}
\bd(F):=	\inf\left\{b\in \R\ \big|\  CF^\lambda_*(F)\cap\partial_{F,J}(CF_*(F))\subset \partial_{F,J} (CF_*^{\lambda+b}(F)),\ \forall \lambda\in \R\right\}.
\end{equation*}
\end{defin}
\noindent The boundary depth satisfies the following stability property \cite{usher2011boundary}: 
\[\bd(F)\leq \int_0^1 \big(\max_{x\in M}F(x,t)-\min_{x\in M} F(x,t)\big)\ dt.\]
In particular, when the Hamiltonian $F$ is $C^0$-small, its boundary depth is small, and hence any boundary chain $\alpha$ in $CF_*(F)$ admits a primitive $\beta$ such that the actions of $\alpha$ and $\beta$ are close.

\section{Spectral killers and the max-inequality}
In this section we explain why the existence of spectral killers implies the max-inequality. In particular, we prove that Proposition~\ref{prop:slow_killer} implies Theorem~\ref{thm:max-ineq}.
\begin{claim}\label{clm:max_ineq}
	Let $F_1,\dots, F_L$ be Hamiltonians supported in pairwise disjoint domains $V_1,\dots,V_L\subset M$ and let $a_i\in QH_*(M)$. Suppose there exist extended domains $\hat V_i\supset V_i$ that are pairwise disjoint, as well as Hamiltonians $K_i$ supported in $\hat V_i$, such that 
	$$
	c(F_i+K_i;a_i)=\val(a_i),\quad \text{and } \|K_i\|_{C^0} =c(F;a_i)-\val(a_i)\text{ for all }i.$$ 
	Then,
	\begin{equation}%\label{eq:max_ineq}
	c(F_1+\cdots+F_L;a_1*\cdots *a_L)\leq\left(\sum_i\val(a_i)\right) + \max_j \{c(F_j;a_j)-\val(a_j)\}.
	\end{equation} 
\end{claim}
\begin{proof}
	The following argument is an adaptation of a proof from \cite{seyfaddini2014spectral} for general quantum homology classes. 
	Using the stability and subadditivity properties of spectral invariants and noticing that the Hamiltonians $\{F_i+ K_i\}$ are all disjointly supported, we have
	\begin{eqnarray*}
		c(F_1+\cdots+F_L;a_1*\cdots*a_L) 
		&\leq& c\left(\sum_i (F_i+ K_i)\ ;\ a_1*\cdots*a_L\right) +\Big\|-\sum_i  K_i\Big\|_{C^0}\\
		&\leq& \sum_i c(F_i+ K_i;a_i) +\Big\|\sum_i K_i\Big\|_{C^0}= \sum_i \val(a_i) +\max_j \|K_j\|_{C^0} \\
		&=& \sum_i \val(a_i) + \max_j\ \{ c(F_j;a_j) - \val(a_j)\}.\qedhere
	\end{eqnarray*}
\end{proof}

\section{Energy bounds for ``up-hill" Floer trajectories}\label{sec:energy_bound}
    Our main goal for this section is proving Theorem~\ref{thm:main_energy_bound}, which gives a lower bound for the energy of Floer trajectories that cross a tubular neighborhood ``against" the direction of the negative gradient of $H$. 
    We start by fixing some notations.
    \begin{notation}
    Let $(M,g)$ be a Riemannian manifold.
    \begin{itemize}
        \item For a smooth curve $\gamma:[a,b]\rightarrow M$, denote
        $$
        \ell_g(\gamma):=\int_a^b|\dot \gamma(t)|_g\ dt \quad\text{and}\quad E_g(\gamma):=\int_a^b|\dot \gamma(t)|_g^2\ dt.$$
        \item For a vector field $X$, a $k$-form $\lambda$ and a subset $N\subset M$, denote
        $$\|X\|_{N,g}:=\sup_{x\in N}|X(x)|_g \quad \text{and} \quad \|\lambda\|_{N,g}:=\sup_{x\in N}|\lambda(x)|_g,$$
        where \[
        |\lambda(x)|_g:= \sup_{v_1,\dots v_k\in T_x M} \frac{|\lambda(v_1,\dots ,v_k)|}{|v_1|_g\cdots |v_k|_g}
        \]
        When $N=M$ we abbreviate to $\|X\|_{g}$ and $\|\lambda\|_{g}$. 
        Moreover, we denote $\|\lambda\|_{C^1,g}:=\|\lambda\|_g+\|d\lambda\|_g$.
    \end{itemize}
    \end{notation}
    
    The proof of Theorem~\ref{thm:main_energy_bound} requires the following lemmas.
    \begin{lemma}[Isoperimetric inequality]\label{lem:isoperimetric_ineq}
        Let $(N,g)$ be a compact Riemannian manifold (possibly with boundary). There exists a constant $C_{iso}>0$ depending only on $(N,g)$, such that for every smooth 1-form $\lambda $ on $N$ and any smooth loop $\gamma:\R/\Z\rightarrow N$ we have 
        \begin{equation*}
            \Big|\int_\gamma \lambda\ \Big|\leq \|\lambda\|_{C^1,g}\cdot C_{iso}\cdot  \ell_g(\gamma)^2.
        \end{equation*}
    \end{lemma}
    \begin{proof}
        The proof of this lemma is a standard isoperimetric-inequality type argument. We include it here for the convenience of the reader. 
        Consider a finite open cover of $N$ by balls and half-balls (near the boundary, if $\partial N\neq \emptyset$) with charts. More formally, let $x_1,\dots,x_d$ be coordinates on $\R^d$ and consider domains $B_i\subset \R^d$, where $d=\dim{N}$, such that for each $i$  either  $B_i=B_0(r_i)$ is an open ball around the origin of radius $r_i$,  or $B_i=B_0(r_i)\cap\{x_1\geq 0\}$. For each $i$, denote 
        \[
        \partial B_i:= \begin{cases}
            \emptyset, &\text{ if }B_i=B_0(r_i),\\
            B_0(r_i)\cap\{x_1=0\}, &\text{ if }B_i=B_0(r_i)\cap\{x_1\geq 0\}.
        \end{cases}
        \] 
        Consider in addition smooth embeddings
        \[
        \varphi_i:(B_i,\partial B_i)\rightarrow (N,\partial N), \qquad  U_i:=\varphi_i(B_i) \quad\text{such that} \quad \cup_i U_i=N.
        \]
        By the Lebesgue lemma there exists $\rho>0$ with the following property. For all $x\in N$ there exists $i$ such that $B_{g}(x,\rho)\subset U_i$, where $B_g(x,\rho):=\{y\in N:d_g(x,y)< \rho\}$ is the ball around $x$ of radius $\rho$, with respect to the metric induced by $g$. We split into two cases, depending on the length of the loop $\gamma$. 

        \vspace{3pt}

        Starting with the case where $\ell_g(\gamma)\geq\rho$, 
        \begin{align}\label{eq:iso_long_loop}
                \Big| \int_\gamma \lambda \Big| &\leq \|\lambda\|_{g}\cdot \int_0^1 |\dot \gamma(t)|_g\ dt 
                =\|\lambda\|_{g}/\rho\cdot \rho\cdot \ell_g(\gamma) \leq \|\lambda\|_{g}/\rho\cdot \ell_g(\gamma)^2.
            \end{align}

        \vspace{3pt}
        Now suppose {$\ell_g(\gamma)<\rho$} and set  $x = \gamma(0)$. Since every path is contained in a ball of radius at most its length, $\gamma\subset B_g(x,\ell_g(\gamma))\subset B_g(x,\rho)$. By our choice of $\rho$, $\gamma\subset B_g(x,\rho)\subset U_i$ for some $i$. Consider the loop $\hat\gamma:=\varphi_i^{-1}\circ\gamma$ in $B_i\subset \R^d$ and the pullback 1-form $\hat \lambda:=\varphi_i^*\lambda$. 
            Let us first prove the lemma for $\hat\gamma$ and $\hat \lambda$. Consider the capping disk
            \[
            a:[0,1]^2\rightarrow B_i\subset \R^d,\qquad a(s,t)=s\hat\gamma(t)+(1-s)\hat\gamma(0).
            \]
            We remark that the image of $a$ is indeed contained in $B_i$ since $B_i$ is convex and $\hat\gamma\subset B_i$. Using Stokes' theorem, we see
            \begin{align}\label{eq:iso_ineq_for_Rd}
                \int_{\hat \gamma}\hat \lambda &= \int_{a}d\hat \lambda 
                = \int_0^1\int_0^1 d\hat \lambda \Big(\frac{\partial a}{\partial s},\frac{\partial a}{\partial t}\Big)\ ds\ dt
                = \int_0^1\int_0^1 d\hat \lambda \Big(\hat \gamma(t)-\hat \gamma(0),s\frac{d}{dt}\hat\gamma(t)\Big)\ ds\ dt \nonumber\\
                &\leq \|d\hat\lambda\|_{g_0} \cdot \int_0^1\int_0^1 s|\hat \gamma(t)-\hat \gamma(0)|_{g_0}\cdot \Big|\frac{d}{dt}\hat\gamma(t)\Big|_{g_0}\ ds\ dt \nonumber\\
                &\leq \|d\hat\lambda\|_{g_0} \cdot \int_0^1\int_0^1 s\ell_{g_0}(\hat\gamma)\cdot \Big|\frac{d}{dt}\hat\gamma(t)\Big|_{g_0}\ ds\ dt \nonumber\\
                &= \|d\hat\lambda\|_{g_0} \cdot\frac{1}{2}\cdot \ell_{g_0}(\hat\gamma)\cdot \int_0^1 \Big|\frac{d}{dt}\hat\gamma(t)\Big|_{g_0}\ dt
                =\|d\hat\lambda\|_{g_0}/{2}\cdot \ell_{g_0}(\hat\gamma)^2.
            \end{align}
            Having proved  the lemma for $\hat\gamma$ and $\hat \lambda$, let us relate the relevant measurements in $B_i\subset\R^d$ to the ones on $N$.
            Starting with the length of $\gamma$ and $\hat\gamma$, we see
            \begin{align}\label{eq:length_gamma_vs_hat}
                \ell_{g_0}(\hat\gamma) &= \int_0^1\Big|\frac{d}{dt}\hat\gamma(t)\Big|_{g_0} dt =\int_0^1 |d\varphi_i^{-1} \dot \gamma|_{g_0} dt \nonumber\\
                &\leq \|d\varphi_i^{-1}\|_{C^0}\cdot \int_0^1 \|\dot \gamma\|_{g} dt = \|d\varphi_i^{-1}\|_{C^0}\cdot \ell_g(\gamma).
            \end{align}
            Moreover,
            \begin{align}\label{eq:int_lambda_vs_hat}
                \Big|\int_{\hat\gamma} \hat\lambda \Big| = 
                \Big|\int_{\varphi_i^{-1}\gamma} \varphi_i^*\lambda \Big| = 
                \Big|\int_{0}^1 \lambda(d\varphi_i  d\varphi_i^{-1}\dot\gamma(t)) \ dt\ \Big| =\Big |\int_\gamma \lambda\Big|.
            \end{align}
            Lastly, a straightforward computation shows that $\|d\hat\lambda\|_{g_0} \leq \|d\lambda\|_{g}\cdot \|d\varphi_i\|_{C^0}^2$.
            Together with (\ref{eq:iso_ineq_for_Rd}), (\ref{eq:length_gamma_vs_hat}) and (\ref{eq:int_lambda_vs_hat}) this yields
            \begin{equation}
                \Big |\int_\gamma \lambda\Big|\leq  \|d\lambda\|_{g}/2\cdot \|d\varphi_i\|_{C^0}^2\cdot \|d\varphi_i^{-1}\|_{C^0}^2\cdot \ell_g(\gamma)^2
            \end{equation}
            Denote 
            \[
            L:= \max_i\|d\varphi_i\|_{C^0}\cdot \|d\varphi_i^{-1}\|_{C^0}. 
            \]
            Then, $ \Big |\int_\gamma \lambda\Big|\leq  \|d\lambda\|_{g} L^2/2\cdot \ell_g(\gamma)^2$.

        \vspace{3pt}

        Overall, setting $C_{iso}:=\max\{L^2/2, 1/\rho\}$ and recalling that $\|\lambda\|_{C^1,g}:= \|\lambda\|_{g}+\|d\lambda\|_{g}$, we conclude that, whatever the length of $\gamma$ is, 
        \begin{equation*}
            \Big |\int_\gamma \lambda\Big|\leq \|\lambda\|_{C^1,g}\cdot C_{iso}\cdot \ell_g(\gamma)^2.\qedhere
        \end{equation*}
    \end{proof}
    
    \begin{lemma}\label{lem:dist_from_periodic}
        Let $(N,g)$ be a compact Riemannian manifold (possibly with boundary) and let $X$ be a  non-vanishing vector field on $N$. There exists $0<\varepsilon_0$, depending on $X$, such that 
       for all $\varepsilon\in(0,\varepsilon_0]$ and every loop $\gamma:S^1\rightarrow N$, we have
        \begin{equation*}
            E_g(\gamma) = \int_0^1|\dot \gamma(t)|_g^2\ dt \leq 5 \int_0^1 |\dot \gamma(t) - \varepsilon X\circ \gamma(t)|_g^2\ dt.
        \end{equation*}
    \end{lemma}    
    \begin{proof}
        Set $\sigma:=(3/2)^{1/4}>1$. For every $x\in N$, there exists $\lambda_x\in T^*_x N$ such that $\lambda_x(X(x))=1$ and $|\lambda_x|_g\cdot |X(x)|_g=1$. 
        There exists a small open neighborhood $U_x$ of $x$ in $N$, and an exact extension of $\lambda_x$ to this neighborhood, such that 
        \begin{equation}\label{eq:exact_extension}
        \|\lambda_x\|_{U_x,g}\cdot \|X\|_{U_x,g}< \sigma, \qquad \lambda_x(X(y))> 1/\sigma, \quad \forall y\in U_x.
        \end{equation}
        The open sets $\{U_x\}_{x\in N}$ cover the compact manifold $N$, and hence there exists a finite subcover $\{U_i\}_{i=1}^m$. Let $\rho$ be the Lebesgue number of this cover, namely, for every $x\in N$, the metric ball $B_g(x,\rho)$ is contained in $U_i$ for some $i$. 
        Set $\varepsilon_0:={\rho}/{(2\|X\|_g)}$ and split into two cases with respect to the length of $\gamma$:
        \begin{enumerate}
            \item \underline{$\ell(\gamma)<\rho$}: In this case, $\gamma$ is contained in a metric ball $B_g(x,\rho)$ around some point $x$, which, by our choice of $\rho$, is contained in $U_i = U_{x_i}$ for some $i$. Let $\lambda_i$ be the exact 1-form extending $\lambda_{x_i}$ such that the inequalities (\ref{eq:exact_extension}) hold. Then,
            \begin{align*}
                \int_0^1 |\dot\gamma - \varepsilon X\circ \gamma|_g\ dt 
                &\geq \frac{1}{\|\lambda_i\|_{U_i,g}}\cdot \int_0^1 \lambda_i(-\dot \gamma+ \varepsilon X\circ \gamma)\ dt\\ 
                &\geq \frac{1}{\|\lambda_i\|_{U_i,g}}\cdot\left( \int_0^1 \lambda_i(-\dot \gamma)\ dt + \int_0^1\lambda_i(\varepsilon X\circ \gamma)\ dt\right) \\
                &= \frac{\varepsilon}{\|\lambda_i\|_{U_i,g}}\cdot \int_0^1 \lambda_i(X\circ \gamma)\ dt \overset{(\ref{eq:exact_extension})}{\geq} \frac{\varepsilon}{\sigma\cdot \|\lambda_i\|_{U_i,g}},
            \end{align*}
            where the equality in the bottom row follows from the exactness of $\lambda_i$.
            By Cauchy-Schwartz inequality,
            \begin{equation}\label{eq:lower_bound_derivative_minus_X}
                  \int_0^1 |\dot\gamma - \varepsilon X\circ \gamma|_g^2\ dt\geq \left(  \int_0^1 |\dot\gamma - \varepsilon X\circ \gamma|_g\ dt \right)^2 \geq \frac{\varepsilon^2}{\sigma^2\cdot \|\lambda_i\|_{U_i,g}^2}\overset{(\ref{eq:exact_extension})}{\geq}\frac{\varepsilon^2\|X\|_{U_i,g}^2}{\sigma^4} .
            \end{equation}
            Recalling that $(a+b)^2\leq 2a^2+2b^2$ for any real numbers $a$ and $b$, we have
            \begin{align*}
                \int_0^1 |\dot\gamma|_g^2\ dt &\leq 
                2\int_0^1 |\dot\gamma - \varepsilon X\circ \gamma|_g^2\ dt + 2\int_0^1 |\varepsilon X\circ \gamma|_g^2\  dt \\
                &\leq  2\int_0^1 |\dot\gamma - \varepsilon X\circ \gamma|_g^2\ dt + 2\varepsilon^2 \|X\|_{U_i,g}^2 \\
                &\overset{(\ref{eq:lower_bound_derivative_minus_X})}{\leq} (2+2\sigma^4)\cdot \int_0^1  |\dot\gamma - \varepsilon X\circ \gamma|_g^2\ dt = 5\int_0^1  |\dot\gamma - \varepsilon X\circ \gamma|_g^2\ dt,
            \end{align*}
            where the last equality follows from our choice of $\sigma = (3/2)^{1/4}$.
            
            \item \underline{$\ell(\gamma)\geq \rho$}: In this case,
            \begin{equation*}
                \rho^2 \leq \ell(\gamma)^2 = \left(\int_0^1 |\dot \gamma|_g\ dt\right)^2\leq \int_0^1|\dot \gamma |_g^2\  dt.
            \end{equation*}
            Since $(a-b)^2\geq \frac{1}{2}a^2 - b^2$, we have
            \begin{align*}
                \int_0^1|\dot \gamma - \varepsilon X\circ \gamma|_g^2\ dt &\geq \frac{1}{2} \int_0^1 |\dot \gamma |_g^2\ dt - \int_0^1| \varepsilon X\circ \gamma|_g^2\ dt               \geq \frac{1}{2} \int_0^1 |\dot \gamma |_g^2\ dt - \varepsilon^2 \| X\|_g^2 \\
                &\geq \frac{1}{2} \int_0^1 |\dot \gamma |_g^2\ dt - \frac{\varepsilon^2 \| X\|_g^2}{\rho^2}\int_0^1 |\dot \gamma |_g^2\ dt\\
                & = \left(\frac{1}{2}  -\frac{\varepsilon^2 \| X\|_g^2}{\rho^2} \right)\int_0^1 |\dot \gamma |_g^2\ dt.
            \end{align*}
            Since $\varepsilon\leq \varepsilon_0 = {\rho}/({2\|X\|_g)}$, we conclude that
            \begin{equation*}
                \int_0^1|\dot \gamma - \varepsilon X\circ \gamma|_g^2\ dt\geq \frac{1}{4} \int_0^1 |\dot \gamma |_g^2\ dt.
            \end{equation*}\qedhere
        \end{enumerate}
    \end{proof}
    
    \begin{rem}
       The constant $\varepsilon_0$ from Lemma~\ref{lem:dist_from_periodic} depends continuously on $X$. More formally, fix $\varepsilon_0'<\varepsilon_0$ and consider a perturbation $X_\delta :=X+ Y$  of $X$, where $Y$ is a possibly time-dependent vector field such that $\|Y\|_{g}\leq \delta$. Consider the 1-forms $\lambda_i$ constructed with respect to $X$, in the proof of  Lemma~\ref{lem:dist_from_periodic}. Inequalities (\ref{eq:exact_extension}) will hold for $X_\delta$ as well, assuming that $\delta$ is small enough. 
       As a result, following the above proof for $X_\delta$ instead of $X$, gives the  inequality 
       \[
       \int_0^1|\dot \gamma(t)|_g^2\ dt \leq 5 \int_0^1 |\dot \gamma(t) - \varepsilon X_\delta\circ \gamma(t)|_g^2\ dt,
       \]
       provided that $\varepsilon\leq \varepsilon_0'$ and that $\delta$ is small enough.
    \end{rem}
    
    \begin{setup}\label{set:energy_bnd}
    Throughout this section we work under the following notations and assumptions:
    \begin{enumerate}
       \item $(M, \omega)$ is a closed symplectic manifold with an almost complex structure $J$, and $g$ is the compatible Riemannian metric, i.e., $g(-,-):=\omega(-,J-)$.
       \item $N\subset M$ is a compact submanifold with boundary, such that $M\setminus N$ has at least two connected components. Let $\partial N = \partial_+ N\bigsqcup \partial_- N$ be a decomposition of the boundary into two components corresponding to disjoint connected components of $M\setminus N$.
        \item $h:M\rightarrow \R$ is a function having no critical points in $N$ {and such that $ h = 0 $ on $ \partial_- N $ and $ h = 1 $ on $ \partial_+ N $}. In particular this implies that $N$ is diffeomorphic to a product of a hypersurface and an interval. To simplify future notation, we assume in addition that the values of $h$ outside of $N$ are disjoint from its values in $N$.
        \item $H:M\times S^1\times \R \rightarrow \R$ is a small perturbation of a homotopy of Hamiltonians that, on $N$ coincides with $\varepsilon\cdot h+\beta(s,t)$ for some $\varepsilon>0$ and $\beta:\R\times S^1\rightarrow\R$. More explicitly, 
        \[
        H|_N(x,t,s)=\varepsilon\cdot h(x)+ h'(x,t,s)+\beta(s,t)
        \]
        where $h':M\times S^1\times \R\rightarrow \R$ is any homotopy such that $\partial_s h'$ is supported in $M\times S^1\times [-R,R]$ for fixed $R$ and $\max\{\|\partial_s h'\|_{C^0}, \|X_{h'}\|_g\} \leq \min\{\delta,\delta/(2R)\}$ for $\delta\ll \varepsilon$ . We stress  that many of the assertions in this section hold under the assumption that $\delta$ is small enough (in particular much smaller than $\varepsilon$).
    \end{enumerate}
    
    \end{setup}
    \begin{lemma}\label{lem:crossing_loops_bound}
        Consider Setup~\ref{set:energy_bnd} and fix $c_\pm\in h(N)$. Let
        $u:\R\times S^1\rightarrow M$ be a solution to the Floer equation with respect to $(H,J)$, and assume there exists $s\in\R$ such that $u(s,-)$ intersects $h^{-1}(c_-)$ and $h^{-1}(c_+)$. Then,
        \begin{equation*}
            \int_0^1 \Big|\frac{\partial u}{\partial s}(s,t)\Big|_g\ dt \geq \frac{c_+-c_-}{{2}\|dh\|_g}\ .
        \end{equation*}
    \end{lemma}
    \begin{proof}
        Since the image of $u(s, -)$ intersects both $h^{-1}(c_-)$ and $h^{-1}(c_+)$ there exist $t_\pm\in[0,1]$ such that $u(s,t_\pm)\in h^{-1}(c_\pm)$ and $u(s,t)\in N$ for all $t$ between $t_-$ and $t_+$. Assume without loss of generality that $t_-<t_+$. Then,
        \begin{align*}
             \int_0^1 \Big|\frac{\partial u}{\partial s}(s,t)\Big|_g\ dt 
             &\geq \int_{t_-}^{t_+}  \Big|\frac{\partial u}{\partial s}(s,t)\Big|_g\ dt  =  \int_{t_-}^{t_+}  \Big|\frac{\partial u}{\partial t}(s,t) - X_H\circ u(s,t)\Big|_g\ dt \\
             &\geq \frac{1}{\|dh\|_g}\cdot  \int_{t_-}^{t_+}  \Big|dh\Big(\frac{\partial u}{\partial t}(s,t)-X_H\circ u(s,t)\Big)\Big|_g\ dt  \\
             &\geq \frac{1}{\|dh\|_g}\cdot  \int_{t_-}^{t_+}  \Big|dh\Big(\frac{\partial u}{\partial t}(s,t)-(\varepsilon X_h+ X_{h'})\circ u(s,t)\Big)\Big|_g\ dt  \\
             &=   \frac{1}{\|dh\|_g}\cdot  \int_{t_-}^{t_+}  \Big|dh\Big(\frac{\partial u}{\partial t}(s,t)\Big)- dh(X_{h'})\Big|_g\ dt  \\      
             &\geq \frac{1}{\|dh\|_g}\cdot  \Big|\int_{t_-}^{t_+}  dh\Big(\frac{\partial u}{\partial t}(s,t)\Big)\ dt \Big| - \|X_{h'}\|_g\\
             &> \frac{1}{\|dh\|_g}\cdot \big|h(u(s,t_+))-h(u(s,t_-))\big|  - \delta= \frac{c_+-c_-}{\|dh\|_g} - \delta.
        \end{align*}
        Clearly, $\delta$ is small enough, the assertion of the lemma holds.
    \end{proof}
    The next lemma is taken from \cite[Lemma 3.6]{hein2012conley}. We briefly repeat its proof, paying a little more attention to constants, since we wish to obtain an energy bound that depends only on the ``model" function $h$ and not on $H$.
    \begin{lemma}[Hein's Usher lemma]\label{lem:Hein_sikorav}
        Consider Setup~\ref{set:energy_bnd}. There exists a constant $C(N,\omega|_N,J|_N, h|_N)$ such that for every Floer trajectory $u$ with respect to $(H,J)$ that intersects both $\partial_- N$ and $\partial_+ N$, it holds that
        \begin{equation*}
            E(u)+m(u^{-1}(N))\geq C(N,\omega|_N,J|_N, h|_N),
        \end{equation*}
        where $m(-)$ is the standard Lebesgue measure on $\R\times S^1$.
    \end{lemma}
    \begin{proof}
        Set $S:=u^{-1}(N)\subset \R\times S^1$. By Setup~\ref{set:energy_bnd}, $H(x,t,s)=\varepsilon h(x)+h'(x,t,s) +\beta(s,t)$ on $N$. Replacing $H$ with $H-\beta$, the corresponding Floer equation does not change (since $X_H$ stays the same), and therefore the Floer trajectories of $H$ and  $H-\beta$ are the same. Therefore, we assume from now on that $H=\varepsilon h + h'$ on $N$. The graph $\tilde u:S \rightarrow N\times \R\times S^1$ of the (restriction to $N$ of the) Floer trajectory $u$ is holomorphic with respect to the almost complex structure \begin{equation*}
            \tilde J(v):=\begin{cases}
                        J(v), & \text{if } v\in TN,\\
                        \partial_t + X_H, & \text{if } v=\partial_s.
                      \end{cases}
        \end{equation*}
        This almost complex structure is tamed by the symplectic form 
        \[\tilde \omega = (ds-dH)\wedge dt - \partial_sH\cdot ds\wedge dt+ \omega = (1-\partial_s H)ds\wedge dt -dH\wedge dt +\omega,\]
        where $dH = \omega(-,X_H)$ is the differential of $H$ is the $M$-directions (see, e.g.  \cite[Section 8.1]{mcduff2012j}). Note that $\tilde \omega$ is symplectic since  $|\partial_sH| = |\partial_s h'|<\delta<1$ when $\delta$ is small. The area of $\tilde u$ is 
        \begin{align}\label{eq:area_of_lifted_traj_h}
            \int_S \tilde u^*\tilde \omega &= \int_{S} ds\wedge dt +\int_S u^*\omega - \int_S u^*dH\wedge dt- \int_S \partial_s H\circ u\ ds\wedge dt \nonumber\\
            &= m(S) +\int_S u^*\omega - \int_S u^*dH\wedge dt - \int_S \partial_s H\circ u\ ds\wedge dt \nonumber\\            &\overset{(\ref{eq:energy_identity_homotopies})}{=} m(S)+ E(u|_S) - \int_S \partial_s H\circ u\ ds\wedge dt \nonumber\\       
            &\leq m(S)+E(u|_S) + 2R\cdot \|\partial_s h'\|_{C^0} \nonumber\\       
            &\leq m(S)+E(u|_S) + \delta.
        \end{align}
        Note that the last inequality above is due to our assumption on the support of $\partial_s h'$ and its norm.
        
        Given (\ref{eq:area_of_lifted_traj_h}), it is sufficient to prove a lower bound for $\int_S \tilde u^*\tilde \omega$, that depends only on $\omega, J$ and $h$, and does not depend on $\varepsilon$. 
        We claim that such a lower bound can be obtained from Sikorav \cite[Proposition 4.3.1(ii)]{sikorav1994some}. 
        To see this, let $g$ be the product Riemannian metric on $N\times \R\times S^1$, namely, $\tilde g(-,-)=ds\wedge dt(-,j-)+ \omega(-,J-)$, where $j$ is the standard complex structure on $\R\times S^1$, restricted to $S$. Let $\Sigma\subset N$ be a hypersurface such that $\partial_\pm N$ lie in different connected components of $N\setminus \Sigma$. For each $x\in \Sigma$, let $r_x$ be the maximal radius of a ball centered at $x$ and contained in $N$. Finally, denote by $r_\star:=\min_{a\in S} r_x$. If $r_\star$ is bigger than the injectivity radius of the cylinder, which is 1, we reduce it to be 1. Since $u$ intersects both $\partial_\pm N$, there exists $z_0\in S$ such that $u_0:=u(z_0)\in \Sigma$. Let $B\subset N\times\R\times S^1$ be a ball of radius $r_\star$ centered at $(u_0,z_0)$ (such a ball can be found in the polydisc of radius $r_\star$, which exists by our choice of $r_\star$).   
        \cite[Proposition 4.3.1(ii)]{sikorav1994some} states that
        \begin{equation}
            E(\tilde u|_{\tilde u^{-1}(B)})\geq c r_\star^2,
        \end{equation}
        where $c$ depends on:
        \begin{itemize}
            \item The isoperimetric constant: $C_1>0$ such that every loop $\gamma$ that is contained in a ball $B(x,r)$ of radius $r\leq r_\star$ bounds a disc in $B(x,r)$ of Riemannian area less than $C_1 \ell_{\tilde g}(\gamma)^2$,
            \item The taming constant: $C_2>0$ such that for every $v\in T_x(\tilde N)$, $|v|_{\tilde g}^2\leq  C_2 \omega_x(v, \tilde J v)$.
        \end{itemize}
        In fact, one can take $c = (4C_1C_2)^{-1}$.
        Recall that the metric $\tilde g$ on $N\times \R\times S^1$ does not depend on $H$ and therefore so does $C_1$.
        As for the taming constant, let us show that it is smaller than 2, when $\varepsilon$ is small enough. That is, for every tangent vector $v$, 
        \begin{equation}\label{eq:taming_prop}
            |v|_{\tilde g}^2 \leq 2\tilde \omega(v,\tilde Jv). 
        \end{equation}
        Write $v=(v_M, v_s,v_t)\in TM \oplus T\R\oplus T S^1$, then $\tilde J v = (-v_t, v_s, v_s X_H-v_t JX_H +Jv_M)$ and 
        \begin{align*}
            \tilde\omega(v, \tilde J v) =&(1-\partial_s H)\cdot( v_s^2+v_t^2) -  dH(v_M)\cdot v_s +dH(v_s X_H-v_t JX_H +Jv_M)\cdot v_t\\
            &+\omega(v_M, v_s X_H-v_t JX_H +Jv_M)\\
            =& (1-\partial_s H)\cdot( v_s^2+v_t^2) -  dH(v_M)\cdot v_s -v_t^2 dH(JX_H) +dH(Jv_M)\cdot v_t\\
            &+v_sdH(v_M)+v_tdH(Jv_M)+\omega(v_M, Jv_M)\\
            =& (1-\partial_s H)\cdot( v_s^2+v_t^2) -v_t^2 \|dH\|_g^2 +2v_t\cdot dH(Jv_M)+\omega(v_M, Jv_M).
        \end{align*}
        Recalling that $dH=\varepsilon dh+ dh'$ with $\|dh'\|_g<\delta$, and  $\max_{M\times S^1\times \R}|\partial_s H|= \max_{M\times S^1\times \R}|\partial_sh'|<\delta$, we see that, when $\varepsilon$ and $\delta$ are small enough,
        \begin{align*}
            \tilde\omega(v, \tilde J v) \geq |v|_{\tilde g}^2 \cdot (1 - \frac{1}{2}).
        \end{align*}
        As a consequence, $C_2 \leq  (1-\frac{1}{2})^{-1}= 2$. We conclude that 
        \[
        m(S)+E(u|_S)\geq E(\tilde u) - \delta  = \int \tilde u ^*\tilde \omega - \delta \geq cr_\star^2  - \delta \geq \frac{c}{2}r_\star^2.
        \]
    \end{proof}
    We will first prove  Theorem~\ref{thm:main_energy_bound} for ``up-hill" trajectories, namely ones that start in $V$ and end outside of $\hat V$.
    \begin{prop}\label{prop:up-hill}
    Consider again Setup~\ref{set:energy_bnd}, and let $V$ be one of the connected components of $M\setminus N$. Denote $\hat V:=\overline{V}\cup N$, $\partial_- N := \partial V$ and $\partial_+ N:=\partial \hat V$. Assume that 
    \[
    h|_V\leq 0, \qquad h|_{\partial V}=0,\qquad  h|_{M\setminus \hat V} \geq 1 \qquad \text{and}\qquad  h|_{\partial \hat V}=1.
    \]
    Then, there exists $\varepsilon_0$ and a constant $C(M, g_J,h)>0$ such that for any $\varepsilon\in (0,\varepsilon_0)$ and for any homotopy of Hamiltonians $H:M\times S^1\times \R \rightarrow \R$ satisfying $H|_N=\varepsilon \cdot h +h'+\beta$ the following holds. Let $u:\R\times S^1\rightarrow M$ be a solution to Floer equation with respect to $(H,J)$, and assume there exist $s_0<s_1\in\R$ such that 
        $$ u(s_0,-)\subset V \quad\text{and}\quad u(s_1,-)\subset M\setminus \hat V.$$ 
    Then 
    \begin{equation*}
            E(u)\geq \frac{1}{36\max\{\|dh\|_g^2,\  C_{iso}(N)\cdot \|dh\circ J\|_{C^1,g}^2\}}
    \end{equation*}
    where $C_{iso}(N)$ is the isoperimetric constant from Lemma~\ref{lem:isoperimetric_ineq}.
    \end{prop}
    \begin{proof}
        Let 
        \begin{align*}
            N_0 &:=h^{-1}(1/3,2/3),\\
            s_0' &:= \max\left\{s\in[s_0,s_1]: u(s,S^1)\subset h^{-1}(-\infty,1/3]\right\},\\
            s_1' &:= \min\left\{s\in[s_0',s_1]: u(s,S^1)\subset h^{-1}[2/3,+\infty)\right\}.
        \end{align*}
        Then, $s_0< s_0'<s_1'<s_1$ and for every $s\in [s_0',s_1']$, the image of $u(s,-)$ intersects $N_0$. Finally, split the interval $[s_0',s_1']$ into the set of loops contained in $N$ and the set of loops that cross a connected component of $N\setminus N_0$:
        \begin{align*}
            A:= \left\{s\in [s_0',s_1']: u(s,S^1)\subset N\right\}, \quad
            B:= [s_0',s_1']\setminus A.
        \end{align*}
        Consider the function
        \begin{equation}\label{eq:def_phi}
            \varphi:[s_0,s_1]\rightarrow\R,\quad \varphi(s):=\int_{S^1} h(u(s,t))\ dt,
        \end{equation}
        and identify $S^1\cong \R/\Z$. Then, 
        \begin{equation*}
            \frac{1}{3} \leq \varphi(s_1') - \varphi(s_0') = \int_{[s_0',s_1']} \frac{d}{ds}\varphi(s)\ ds = \int_A \frac{d}{ds}\varphi(s)\ ds +\int_B \frac{d}{ds}\varphi(s)\ ds.     
        \end{equation*} 
        Therefore one of the summands of the RHS is at least 1/6. Let us split into cases:
        \begin{enumerate}
            \item \underline{$ \int_A \frac{d}{ds}\varphi(s)\ ds\geq 1/6$}:
            The derivative of $\varphi(s)$ is given by
        \begin{align*}
            \frac{d}{ds}\varphi(s)&= \int_0^1 dh\Big(\frac{\partial u}{\partial s}(s,t)\Big)\ dt
            =  \int_0^1 dh\Big(-J\frac{\partial u}{\partial t}(s,t) + JX_H\circ u(s,t)\Big)\ dt\\
            &=  -\int_0^1 dh\circ J\Big(\frac{\partial u}{\partial t}(s,t)\Big)\ dt + \int_0^1 dh\Big(-\nabla_J H\Big)\circ u(s,t)\ dt.
        \end{align*}
        When $s\in A$, $u(s,t)\in N$ for all $t$, and we have $H\circ u(s,t) = (\varepsilon h+h') \circ u(s,t)$. In particular $dh\Big(-\nabla_J H\Big)\circ u(s,t)=dh\Big(-\varepsilon\nabla_Jh- \nabla_J h' \Big)\circ u(s,t)\leq 0$ when $\|\nabla_J h'\|<\delta$ is small enough. Together with the isoperimetric inequality stated in Lemma~\ref{lem:isoperimetric_ineq}, applied to $\lambda = dh\circ J$, this implies  that
        \begin{equation*}
            \frac{d}{ds}\varphi(s)\leq -\int_0^1 dh\circ J\Big(\frac{\partial u}{\partial t}(s,t)\Big)\ dt \leq {\|\lambda\|_{C^1,g}}C_{iso}(N) \cdot \int_{0}^1 \Big|\frac{\partial u}{\partial t}\Big|_g^2\ dt, \quad \text{for all } s\in A.
        \end{equation*}
        We now apply Lemma~\ref{lem:dist_from_periodic} to obtain
        \begin{align*}
            \int_{0}^1 \Big|\frac{\partial u}{\partial s}\Big|_g^2\ dt =& \int_{0}^1 \Big|\frac{\partial u}{\partial t} - X_H\circ u(s,t)\Big|_g^2\ dt = \int_{0}^1 \Big|\frac{\partial u}{\partial t}- \varepsilon X_h\circ u(s,t)\Big|_g^2\ dt\\
            \overset{Lemma~\ref{lem:dist_from_periodic}}{\geq}&\frac{1}{5}\int_{0}^1 \Big|\frac{\partial u}{\partial t}\Big|_g^2\ dt \geq  \frac{1}{5\|\lambda\|_{C^1,g}C_{iso}(N)}\cdot\left( \frac{d}{ds}\varphi(s)\right), \quad \text{for all }s\in A.
        \end{align*}
        Integrating  over $s\in A $ and recalling that $\lambda:= dh\circ J$, we get
       \begin{align*}
            E(u)&\geq \int_{A} \int_{0}^1 \Big|\frac{\partial u}{\partial s}\Big|_g^2\ dt\ ds \geq \frac{1}{5\|dh\circ J\|_{C^1,g}C_{iso}(N)} \int_{A}\frac{d}{ds}\varphi(s)\ ds\\
            &\geq  \frac{1}{30\|dh\circ J\|_{C^1,g}C_{iso}(N)},
        \end{align*}
        where in the last inequality we used our working assumption for this case, which is $\int_A \frac{d}{ds}\varphi(s)\ ds \geq 1/6$.
        
        \item \underline{$ \int_B \frac{d}{ds}\varphi(s)\ ds\geq 1/6$}:
        Now let us consider $s\in B$. By definition of $B$, the loop $u(s,-)$ intersects both $N_0=h^{-1}([1/3,2/3])$ and $M\setminus N$. Therefore it must intersect both $h^{-1}(c_-)$ and $h^{-1}(c_+)$ for $(c_-, c_+)=(0,1/3)$ or $(c_-, c_+)=(2/3,1)$. In any case Lemma~\ref{lem:crossing_loops_bound} implies that 
        \begin{equation*}
            \int_0^1 \Big|\frac{\partial u}{\partial s}(s,t)\Big|_g \ dt\
            \geq \ \frac{1}{6\|dh\|_g}.
        \end{equation*}
        Integrating the above inequality (squared) over $s\in B$ and using the Cauchy-Schwartz inequality we get
        \begin{equation}\label{eq:lowerbound_times_mB}
            \frac{m(B)}{36\|dh\|_g^2}\leq \int_B\left(\int_0^1 \Big|\frac{\partial u}{\partial s}(s,t)\Big|_g \ dt\right)^2\ ds \leq \int_B\int_0^1 \Big|\frac{\partial u}{\partial s}(s,t)\Big|_g^2 \ dt\ ds \leq E(u),
        \end{equation}
        where $m(B)$ is the Lebesgue measure of $B$.
        On the other hand, for every $s\in B$ we have
        \begin{equation*}
            \frac{d}{ds}\varphi(s) = \int_0^1 dh\left(\frac{\partial u}{\partial s}(s,t)\right) \ dt \leq \|dh\|_g\cdot \int_0^1\Big|\frac{\partial u}{\partial s}(s,t)\Big|_g\ dt.
        \end{equation*}
        As before, we integrate the above inequality  and use the Cauchy-Schwartz inequality to obtain
        \begin{align*}
            \frac{1}{6}&\leq \int_B\frac{d}{ds}\varphi(s)\ ds \leq \|dh\|_g \cdot \int_B\int_0^1 \Big|\frac{\partial u}{\partial s}(s,t)\Big|\ dt\ ds\\
            &\leq \|dh\|_g\cdot \left(m(B)\cdot \int_B\int_0^1\Big|\frac{\partial u}{\partial s}(s,t)\Big|_g^2\ dt\right)^{\frac{1}{2}}.
        \end{align*}
        Rearranging the above we find
        \begin{equation}\label{eq:lowerbound_1_over_mB}
            E(u)\geq \int_B\int_0^1\Big|\frac{\partial u}{\partial s}(s,t)\Big|_g^2\ dt\geq \frac{1}{36\|dh\|_g^2}\cdot\frac{1}{m(B)}.
        \end{equation}
        By multiplying inequalities (\ref{eq:lowerbound_times_mB}) and (\ref{eq:lowerbound_1_over_mB}) and then taking a square root we obtain a lower bound that is independent of the measure of $B$:
        \begin{equation*}
            E(u)\geq \frac{1}{36\|dh\|_g^2}.
        \end{equation*}
        \end{enumerate}
        Combining the lower bounds found in case 1 and case 2, it is not difficult to see that in any situation the energy of $u$ is bounded by
        \begin{equation*}
            E(u)\geq \frac{1}{36\max\{\|dh\|_g^2,\  C_{iso}(N)\cdot \|dh\circ J\|_{C^1,g}^2\}}.
        \end{equation*}
        Here we used the fact that $\|dh\circ J\|_{C^1,g}\geq \|dh\circ J\|_g = \|dh\|_g$.
    \end{proof}
    We are now ready to prove Theorem~\ref{thm:main_energy_bound}.
    \begin{proof}[Proof of Theorem~\ref{thm:main_energy_bound}]
        Let $h:{N}\rightarrow\R$ be a function without critical points taking values 0 and 1 on $\partial V$ and $\partial \hat V$ respectively. As in Setup~\ref{set:energy_bnd}, let $H:M\times S^1\times \R \rightarrow\R$ be a Hamiltonian homotopy such that $H|_N=\varepsilon \cdot h+h'+\beta$ for some $\varepsilon>0$, $\beta:\R\times S^1\rightarrow\R$ and  $h':M\times S^1\times \R\rightarrow \R$ with $\max\{\|\partial_s h'\|_{C^0},\|X_{h'}\|_g\}<\min\{\delta, \delta/2R\}$ for $\delta>0$ small. 
        Let $u:\R\times S^1\rightarrow M$ be a solution to Floer equation with respect to $(H,J)$ connecting $x_-$ and $x_+$. Clearly, if $x_-\subset V$ and $x_+\subset M\setminus \hat V$ then Proposition~\ref{prop:up-hill} yields the required lower bound. Therefore it remains to deal with the case where $x_\pm$ both lie in $V$ or in $M\setminus \hat V$. Assume that $x_\pm\subset V$, the other case is completely analogous. Recall that by assumption, $u$ intersects $M\setminus \hat V$. Denote by 
        $$
        N_-:=h^{-1}(0,1/3),\quad  N_0:= h^{-1}(1/3,2/3), \quad N_+:=h^{-1}(2/3,1) \quad  \text{and}\quad
        V':=\overline{V\cup N_-}.
        $$
        We split into two cases:
        \begin{enumerate}
            \item \underline{$\exists s_1\in \R,\ u(s_1,-)\subset M\setminus V'$}: Choosing $s_0<s_1$ small enough such that $u(s_0,-)\subset V$, we may apply Proposition~\ref{prop:up-hill} to $u$ with respect to $N_-$ instead of $N$ and $3h$ instead of $h$. The lower bound we obtain this way is 1/9 times the one stated in Proposition~\ref{prop:up-hill}.
            \item \underline{$\forall s\in\R, u(s,-)$ intersects $V'$}: Consider the set of $s$-values for which $u(s,-)$ intersects $N_+$: 
            $$
            Z:=\{s\in \R: u(s,-)\cap N_+ \neq \emptyset\}.
            $$
            Then, for every $s\in Z$, $u(s,-)$ crosses $N_0$ (otherwise we would fall into case 1). Arguing as in \cite{hein2012conley}, let us split again into cases, depending on whether the measure of $Z$ is big or small. More formally, let $C$ be the constant from Lemma~\ref{lem:Hein_sikorav} applied to $N_0$, and split into the following cases:
            \begin{enumerate}
                \item \underline{$m(Z)\leq C/2$}: Denote $S:=u^{-1}(N_0)$, then $S\subset Z$. By Lemma~\ref{lem:Hein_sikorav}, $C\leq E(u)+m(S)\leq E(u)+m(Z)\leq E(u)+C/2$. We conclude that $E(u)\geq C/2$.
                
                \item \underline{$m(Z)>C/2$}: Since  
                $E(u)\geq \int_Z\int_{S^1} |\partial_s u|_g^2 \ dt\ ds$ and $m(Z)$ is bounded below, it sufficient to prove a lower bound for the integral of $|\partial_s u |_g^2$ over $t$. Fix $s\in Z$, then $u(s,-)$ intersects both connected components of $\partial N_0$, on which $h$ equals ${1/3}$ and $2/3$. By Lemma~\ref{lem:crossing_loops_bound}, together with the Cauchy-Schwarz inequality, we have
                \[
                \int_0^1 |\partial_s u(s,t)|_g^2\ dt \geq \left(\int_0^1 |\partial_s u(s,t)|_g\  dt\right)^2\geq \frac{1}{36\|dh\|_g^2}. 
                \]
                Overall, $E(u)\geq m(Z)/(36\|dh\|_g^2) \geq C/(72\|dh\|_g^2)$. \qedhere 
                \end{enumerate}
        \end{enumerate} 
    \end{proof}

    \section{The Floer complex for locally supported Hamiltonians}
    Our main goal for this section is to prove Proposition~\ref{prop:slow_killer}, that is, construct an appropriate spectral killer for $F$. The proof uses the energy bound stated in Theorem~\ref{thm:main_energy_bound} to show that certain moduli spaces, counted by continuation maps or differentials, are empty.

    \begin{rem}[perturbations]\label{rem:choice_of_perturbations}
    In order to use Theorem~\ref{thm:main_energy_bound}, we fix a ``model function" $h$ on $N$, and an almost complex structure $J\in \cJ_{\operatorname{reg}}$ on $M$ such that $C(N,g_J,h)\geq \gw(N)-\d$ for small $\d>0$ (we note that the exact value of $\d$ will be determined in the proof of Proposition~\ref{prop:slow_killer}). We also fix $0<\varepsilon<\varepsilon_0$ where $\varepsilon_0$ is the parameter from  Theorem~\ref{thm:main_energy_bound} for our choice of $h$ and $J$.  Throughout the section, all non-degenerate perturbed Hamiltonians and homotopies coincide with $\varepsilon h$ up to a function depending only on $s,t$, and up to a $C^\infty$-small perturbation of the form $h':M\times S^1\times \R\rightarrow\R$. We assume that $\partial_sh'$ is supported in $M\times S^1\times [-R,R]$ for fixed $R>0$ and that 
    \begin{equation}\label{eq:h_prime_bounds}
        \max\{ \|\partial_s h'\|_{C^0},\ \|X_{h'}\|_g \} \leq \max\{\delta ,\delta/(2R)\},
    \end{equation}
    where $\delta$ is much smaller than $\varepsilon$. We remark that the latter perturbation is required for achieving regularity of the moduli spaces counted by continuation maps and differentials.
    \end{rem}

    The main ingredient in the construction of a spectral killer as stated in  Proposition~\ref{prop:slow_killer} is the following proposition, which guarantees the existence of a minimal action representative with low actions outside of $V$.
        
    \begin{prop}[low actions representative]\label{prop:low_actions_outside}
    Let $F$ be a non-negative Hamiltonian supported in $V$, fix $\d>0$ and let $a\in QH_*(M)$ be a quantum homology class such that 
    \begin{equation}
        c(F;a)<\gw(N)+\val(a)-\d. 
    \end{equation}
    There exists a perturbation $f$ of $F$, such that the pair $(f,J)$ is Floer regular and a representative $\alpha\in CF_*(f)$ of $a$ such that
    \begin{equation}
        \cA_f(\alpha)\leq c(F;a)+\d 
        \quad\text{and}\quad
        \cA_f(\pi_{V^c}\alpha)\leq \val(a)+\d.
    \end{equation}
    \end{prop}
    Before discussing the proof of the above proposition, let us explain how to derive Proposition~\ref{prop:slow_killer} from it.
    \begin{proof}[Proof of Proposition~\ref{prop:slow_killer}]
       We start by recalling the statement of the theorem. Suppose $F\geq 0$ is a Hamiltonian supported in a domain $V$ and let $N$ be a tubular neighborhood of the boundary of $V$ in $M\setminus V$, as discussed in Section~\ref{sec:results}. Let $a\in QH_*(M)$ be a quantum homology class and assume 
        $$
        0<c(F;a)-\val(a)<\gw(N).
        $$ 
        Let $\d>0$ be small enough, such that in particular $c(F;a)-\val(a)<\gw(N)-\d$. We need to construct a Hamiltonian $K:M\rightarrow \R$ supported in $\overline{V}\cup N=:\hat V$ such that 
        $$
        \|K\|_{C^0} = c(F;a)-\val(a) \quad\text{ and }\quad c(F+K;a)  = \val(a).
        $$
        In what follows we will actually construct a spectral killer supported in a small neighborhood of $\hat V$. To have a spectral killer that is supported in $\hat V$ one simply needs to shrink $N$ a little, and use Remark~\ref{rem:continuity_of_w}. The spectral killer we construct is a smooth approximation of a negative multiple of the indicator function of $\hat V$:
        \begin{equation*}
            K(x) :=  \begin{cases}
                        -c(F;a)+\val(a), & \text{ on }\hat V,\\
                        0 & \text{ outside of } \cN(\hat V).
                    \end{cases}
        \end{equation*}     
        Here $\cN(\hat V)$ denotes an arbitrarily small neighborhood of $\hat V$.  Near the boundary of $\hat V$ the function $K$ is a smooth interpolation between its values in  $\hat V$ and outside.   Clearly, $K$ is supported in $\cN(\hat V)$ and its uniform norm is $c(F;a)-\val(a)$. Therefore it remains to show that $c(F+K;a) = \val (a)$. To that end, consider   the pair $(f,J)$  of a perturbation of $F$ and an almost complex structure, and the representative $\alpha\in CF_*(f)$ of $a\in QH_*(M)$ from Proposition~\ref{prop:low_actions_outside}. Then,  \begin{equation*}
        \cA_f(\alpha)\leq c(F;a)+\d 
        \quad\text{and}\quad
        \cA_f(\pi_{V^c}\alpha)\leq \val(a)+\d.
        \end{equation*}
        Let $H$ be a regular perturbation of the homotopy $(x,t,s)\mapsto F(x,t)+ \beta(s)\cdot K$, 
        where $\beta:\R\rightarrow\R$ is a non-decreasing function that is equal to $0$ for $s\leq 0$ and to 1 for $s\geq 1$. 
        More explicitly,
        \begin{equation*}
            H(x,t,s) = F(x,t)+\beta(s) K+\varepsilon h(x) + h'(x,t,s).
        \end{equation*}
        The perturbation  $\varepsilon h+h'$ is chosen as in Remark~\ref{rem:choice_of_perturbations}. More explicitly, we assume that $ h : M \rightarrow \mathbb{R} $ has no critical points on $ N $, and we have $ h = 0 $ on $ \partial V $ and $ h = 1 $ on $ \partial \hat V $. Furthermore, we  assume that $\partial_s h'$ is supported in $M\times S^1\times [-R,R]$ for fixed $R$, and that $h'$ satisfies (\ref{eq:h_prime_bounds}) for small enough $\delta$.
    Note that $F$ is supported in $V$ and $K$ is constant on $\hat V$. In  particular, $F$ and $K$ are constant on $N$, and thus the homotopy $H$ is of the form considered in Setup~\ref{set:energy_bnd}, and for which Theorem~\ref{thm:main_energy_bound} applies (see also Remark~\ref{rem:main_energy_bound}.\ref{itm:energy_bnd_perturbation}). The left end $H_-$ of this homotopy is precisely the perturbation $f$ of $F$ that is considered in Remark~\ref{rem:choice_of_perturbations} and Proposition~\ref{prop:low_actions_outside}. Moreover, we choose the perturbation $h'$ such that $H_+$, which approximates $F+K$, has the same periodic orbits as $f$ in $V$ and in $M\setminus \cN(\hat V)$, and has no periodic orbits in $N$ (see Remark~\ref{rem:achiving_regularity}). Note that $H_+$ might have periodic orbits in $\cN(\hat V)\setminus\hat V$, where $f$ does not. 
        For a (regular)  such homotopy consider the induced continuation map $\Phi:CF(f)\rightarrow CF(H_+)$. Since $\Phi$ induces an isomorphism on homology, $\Phi(\alpha)$ also represents the class $a$. Let us show that $\cA_{H_+}(\Phi(\alpha))\leq \val(a)+2\d$, which will imply that $c(F+K;a)\leq \val(a)$ when we take $\d$ to zero.
        To show this, let us analyze the energies of Floer trajectories $u:\R\times S^1\rightarrow M$, counted by $\Phi$, starting from capped orbits $(x_-,D_-)$ in $\alpha$ and ending at some $(x_+,D_+)\in \Phi(\alpha)$. We split into cases:
        \begin{enumerate}
            \item \underline{$u\subset \hat V$}: In this region, $H$ is of the form 
            \[
            F(x,t)+\beta(s) \big(-c(F;a)+\val(a)\big) +\varepsilon h(x) +h'(x,t,s).
            \]
            The energy identity for continuation trajectories (\ref{eq:energy_id_homotopies}) reads:
            \begin{align*}
                \cA_{f}(x_-,D_-) - \cA_{H_+}(x_+,D_+) =& E(u)-\int_{\R\times S^1} \partial_s H\circ u \ ds\ dt \\
                \geq&\ 0- \big(-c(F;a)+\val(a)\big)\int_{\R\times S^1}\frac{d}{ds} \beta(s)\ ds\ dt \\
                &+\int_{\R\times S^1} \partial_s h'\circ u\ ds\ dt\\
                \geq&\  c(F;a)-\val(a) - \d,
            \end{align*}
            where the last inequality follows from our assumptions on the support and uniform norm of $\partial_s h'$.
            We have $\cA_{f}(\alpha) \leq c(F;a)+\d$, we conclude that $\cA_{H_+}(x_+,D_+)\leq \val(a)+2\d$ in this case.
            \item \underline{$x_-\subset M\setminus \cN(\hat V)$}: Since $\cA_{f}(\pi_{V^c}\alpha)\leq \val(a)+\d$, and $\partial_s H\leq 0+\partial_s h'$ everywhere. Recalling our assumptions on the norm and support of $\partial_s h'$ and viewing again the energy identity, we conclude that $\cA_{H_+}(x_+, D_+)\leq \val(a)+2\d$.
            \item \underline{$x_-\subset V$ and $u$ intersects both $V$ and $M\setminus \hat V$}: As explained in Remark~\ref{rem:choice_of_perturbations}, the pair $(f,J)$ from Proposition~\ref{prop:low_actions_outside} 
            satisfies the assumptions of Theorem~\ref{thm:main_energy_bound} (together with Remark~\ref{rem:main_energy_bound}.\ref{itm:energy_bnd_perturbation}).  Theorem~\ref{thm:main_energy_bound} guarantees that the energy of any Floer trajectory of $(f,J)$ that starts in $V$ and crosses $N$ is bounded by 
            \[
            E(u)\geq  C(N, g_J,h)\geq \gw(N)-\d.
            \] 
            Recall that we chose $\d$ such that $w(N)-\d > c(F;a)-\val(a)$. Therefore,
            \begin{align*}
                \cA_{H_+}(x_+,D_+)&\leq \cA_f(x_-,D_-) - E(u)+\int_{\R\times S^1} \partial_s H\circ u \ ds\ dt \\
                &\leq c(F;a) - (\gw(N)-\d) + \int_{\R\times S^1}\partial_s h'\circ u\ ds\ dt \\
                &\leq c(F;a) - (c(F;a)-\val(a))+\d = \val(a)+\d.
            \end{align*}
        \end{enumerate}
        We conclude that all capped orbits $(x_+,D_+)$ in $\Phi(\alpha)$ have action bounded by $\val(a)+\d$. When $\d$ tends to zero, $H_+\rightarrow F+K$ and this guarantees that $c(F+K;a)\leq \val(a)$ as required.
    \end{proof}
    The rest of this section is dedicated to proving Proposition~\ref{prop:low_actions_outside}. The proof  requires two lemmas that analyze how the action filtration behaves with respect to the linear projections of the Floer chain complex to $V$ and $V^c$. In what follows we use Notations~\ref{not:CF_element_contained}.

    \begin{lemma}\label{lem:proj_and_cont_or_diff}
        Let $F$ be a non-negative Hamiltonian supported in $V$ and fix $\d>0$. There exists a perturbation $f$ of $F$ and a homotopy $H$ from $f$ to a small Morse function $H_+$, such that $f$ and $H_+$ coincide up to second order on their critical points in $M\setminus V$, the pairs $(f,J)$ and $(H,J)$ are Floer regular and the following holds. 
        \begin{enumerate}
            \item \label{itm:cont_minus_is} Let  $\Phi:CF_*(f)\rightarrow CF_*(H_+)$ be the continuation map associated to $(H,J)$. Then $\Phi$ does not increase action by more than $\d$, and the map 
            $$
            \pi_{V^c}\circ \Phi - \pi_{V^c} 
            $$
            decreases the action by at least $\gw(N)-\d$.
            \item Let $\partial_{H_+}$ be the differential with respect to $(H_+,J)$, then $\pi_{V^c}\circ \partial_{H_+}\circ \pi_{V^c} = \pi_{V^c}\circ \partial_{H_+}$.
        \end{enumerate}
    \end{lemma}
    \begin{proof}
        Consider the linear homotopy $H':M\times S^1\times \R\rightarrow \R$ between $F$ and zero, i.e. $H'(x,t,s):=\beta(s)F(x,t)$, {where $\beta : \mathbb{R} \rightarrow [0,1] $ is a smooth monotone non-increasing function which equals to 1 near $ -\infty $ and to $ 0 $ near $ +\infty $}. Then $H'$ is supported in $V$. Let $h:{N}\rightarrow\R$ be the function from Remark~\ref{rem:choice_of_perturbations}, i.e.  $Crit(h)=\emptyset$, $h|_{\partial V}=0$, $h|_{\partial \hat V}=1$ and $C(N, g_J, h)\geq \gw(N)-\d$. Let $H''$ be a perturbation of $H'$ such that 
        \begin{itemize}
            \item $H''|_N=\varepsilon\cdot h$ for $\varepsilon>0$ small enough,
            \item $H''_-$ is a non-degenerate Hamiltonian and $H''_+$ is a $C^2$-small Morse function on $M$ that does not depend on $t$.
            \item $H''|_{M\setminus V}$ is independent of $t$ and $s$.
        \end{itemize}
        Let $H(x,s,t)=H''(x,t,s)+ h'(x,t,s)$ be a further perturbation of $H''$ such that $(H,J)$ is regular and $H_\pm$ and $H''_\pm$ coincide up to second order on their periodic orbits (see Remark~\ref{rem:achiving_regularity}). Denote $f:=H_-$ and let  $\alpha\in CF_*(f)$ be a chain. 
        \begin{enumerate}
            \item 
            The energy identity (\ref{eq:energy_id_homotopies}) implies that the action of  $\Phi(\alpha)$ is not greater than $\cA_f(\alpha)+\int_{\R\times S^1}\max_M \partial_s H ds\ dt$.
            Our assumption that $F$ is non-negative implies that $\partial_s H' = \left(\frac{d}{ds}\beta(s)\right)\cdot F\leq 0$, namely, $H'$ is a monotone decreasing homotopy. 
            Since $H$ is a small perturbation of $H'$ and the support of $\partial_s H$ is uniformly bounded, we can guarantee that $\int_{\R\times S^1}\max_M\partial_s H\ ds\ dt\leq \d$ when the perturbation is small enough. Therefore $\Phi$ does not increase action by more than $\d$.

            Let $(x,A)$ be a pair of a constant orbit $x\in M\setminus V $ and a capping sphere $A\in \pi_2(M)$. Assume that $\cA_{H_+}(x,A)> \cA_f(\alpha)- w(N)+2\d$, and let us show that $\left<\Phi(\alpha),(x,A)\right> = \left<\alpha,(x,A)\right>$.
            By definition of $\Phi$, it is sufficient to show that for every  $(y,D)\in \alpha$, the count of the moduli space $\cM_{(H,J)}((y,D),(x,A))$ is 1 if $(y,D)=(x,A)$ and zero otherwise. We start by showing that all elements of the above moduli space are contained in the interior of $M\setminus V$. 
            %and a Floer trajectory $u$ connecting $(y,D)$ and $(x,A)$. 
            Indeed, recall that by the energy identity (\ref{eq:energy_id_homotopies}),
            \begin{align*}
            E(u)&= \cA_f(y,D)-\cA_{H_+}(x,A) +\int_{\R\times S^1}\partial_s H\circ u\ ds\ dt\\
            &\leq \cA_f(\alpha)-\cA_{H_+}(x,A) +\int_{\R\times S^1}\partial_s H\circ u\ ds\ dt\\
            &<w(N)-2\d +\int_{\R\times S^1}\partial_s H\circ u\ ds\ dt.
            \end{align*}
            As mentioned above, the integral $\int_{\R\times S^1}\partial_s H\circ u\ ds\ dt$ is bounded by $\d$. Therefore, 
            \[
            E(u)< \gw(N)-\d.
            \]
            On the other hand, by Remark~\ref{rem:choice_of_perturbations} and Theorem~\ref{thm:main_energy_bound}, any solution $u$ that intersects both $\partial V$ and $\partial \hat V$ must have energy at least $w(N)-\d$. 
            This implies that $u\subset int(M\setminus V)$ for all $u\in \cM_{(H,J)}((y,D),(x,A))$. Recall that Lemma~\ref{lem:almost_const_homotopy} above states that when a regular homotopy $H$ is almost constant on a open set $U$, and all elements of a given moduli space are contained in $U$, then  the count is the same as the constant homotopy. Applying this to $U=int(M\setminus V)$ we obtain 
            \begin{align*}
               \left<\Phi(y,D),(x,A)\right>&:= \#\cM_{(H,J)}((y,D),(x,A)) \\
               &= \#\cM_{(H_-,J)}((y,D),(x,A)) = \begin{cases}
                0, & \text{ if }(y,D)\neq(x,A)\\
                1, & \text{ if }(y,D)=(x,A)
            \end{cases}\\
            &=:\left<(y,D),(x,A)\right>.
            \end{align*}
            In other words, $\left<\Phi(y,D)-(y,D),(x,A)\right> =0 $ for every $(y,D)\in \alpha$ and every $(x,A)$ in $M\setminus V$ of action at least $\cA_f(\alpha)-w(N)+2\d$. Hence, we conclude that $\left<\Phi(\alpha)-\alpha,(x,A)\right>$ vanishes for any such $(x,A)$. This immediately implies that $\pi_{V^c}(\Phi(\alpha)-\alpha)$ consists only of generators of action less than $\cA_f(\alpha)-w(N)+2\d$,  and thus the map $\pi_{V_c}\Phi-\pi_{V^c}$ reduces action by at least $\gw(N)-2\d$.
            Since $\d>0$ is arbitrary, this proves part \ref{itm:cont_minus_is} of the lemma.

            \item $H_+$ is a small Morse function and therefore its Floer differential $\partial_{H_+}$ coincides with the Morse differential, extended linearly over $\Lambda$. Since the gradient of $H_+$ on $\partial V$ points outwards of $V$, it follows that there are no Morse trajectories from $V$ to $M\setminus V$.
        \end{enumerate}
    \end{proof}
    
    \begin{rem}[Lemma~\ref{lem:proj_and_cont_or_diff} with virtual counts]
        To prove Lemma~\ref{lem:proj_and_cont_or_diff} with virtual counts as in Remark~\ref{rem:foundations}, one needs to consider all stable maps connecting $(x,A)\in\pi_{V^c}\circ \Phi(\alpha)$ with $(y,D)\in\alpha$, such that the action difference of these two generators is less than $\gw(N)-2\d$. One can show that all such stable maps are contained in $M\setminus V$. Indeed, consider a hypersurface $\Sigma\subset N$ splitting $N$ into two tubular neighborhood $N_\pm$. If there exists a sphere bubble that intersects $\Sigma$, then its energy is at least $\gw(N)$ (see \cite{sikorav1994some} and the proof of Lemma~\ref{lem:Hein_sikorav}). If the stable map intersects $V$ but no sphere bubbles intersect $\Sigma$, then there must be a Floer trajectory intersecting $\Sigma$. This Floer trajectory must cross $N_+$ or $N_-$, since there are no periodic orbits in $N$. As a consequence, we again obtain a lower bound for the energy\footnote{Taking this approach one might need to shrink the value of $\gw(N)$ by a constant factor.}. Overall we conclude that such a stable map has energy that is too big given the action difference.
        Having concluded that all stable maps are contained in $M\setminus V$, one uses the assumptions stated in Remark~\ref{rem:foundations} (see Remark~\ref{rem:lemma_in_perliminaries} as well).
    \end{rem}
    
    \begin{lemma}\label{lem:action_non_dec_out}
        Let $H$ and $\d>0$ be as in Lemma~\ref{lem:proj_and_cont_or_diff}. For any chain $\alpha\in CF_*(f)$ of action less than $\gw(N)+\val(a)-\d$ it holds that 
        \begin{equation}\label{eq:action_bnd_out_gen_chain}
            \cA_f(\pi_{V^c}\alpha) \leq \max\{\cA_{H_+}(\pi_{V^c}\Phi(\alpha)),\val(a)\},
        \end{equation}
        where $\Phi:CF_*(f)\rightarrow CF_*(H_+)$ is the continuation map associated to $(H,J)$.
    \end{lemma}
    \begin{proof}
       By Lemma~\ref{lem:proj_and_cont_or_diff}, the map $\pi_{V^c}\circ \Phi - \pi_{V^c} $ decreases action by at least $\gw(N)-\d$, namely: 
       $$\cA_{H_+}(\pi_{V^c}\circ \Phi(\alpha) - \pi_{V^c}\alpha)\leq \cA_f(\alpha) - \gw(N)+\d\leq \val(a),
       $$
       where the last inequality follows from out assumption on the action of $\alpha$.
       Hence,
       \begin{align*}
        \cA_f(\pi_{V^c}\alpha) &= \cA_{H_+}(\pi_{V^c}\alpha)\leq \max\big\{ \cA_{H_+}(\pi_{V^c}\Phi\alpha),\ \cA_{H_+}((\pi_{V^c}\circ \Phi - \pi_{V^c})\alpha)\big\}\\
        &\leq \max\big\{\cA_{H_+}(\pi_{V^c}\Phi\alpha),\  \val(a)\big\}. 
       \end{align*}
   \end{proof}
   We are now ready to prove Proposition~\ref{prop:low_actions_outside}.
    \begin{proof}[Proof of Proposition~\ref{prop:low_actions_outside}]
       Let $F$ be a non-negative Hamiltonian supported in $V$ and fix $a\in QH_*(M)$. By shrinking $\d$, we may assume that $c(F;a)<\gw(N)+\val(a)-4\d$. We need to construct a representative of $a$ whose total action is bounded by $c(F;a)+\d$ and whose actions in $M\setminus V$ are bounded by $\val(a)+\d$. Clearly, this is trivial if $c(F;a)\leq \val(a)$, so assume otherwise.

        \vspace{3pt}
        
        Let  $H:M\times S^1\times \R\rightarrow \R$ be the homotopy between a perturbation $f$ of $F$ and a $C^2$-small, time independent Morse function, from Lemma~\ref{lem:proj_and_cont_or_diff}. Let $\alpha\in CF_*(f)$ be a cycle representing the class $a\in QH_*(M)$ whose action is less than $c(F;a)+\d\leq \gw(N)+\val(a)-3\d$. Denote by $\Phi:CF_*(f)\rightarrow CF_*(H_+)$  the continuation map with respect to $(H,J)$, then $\Phi(\alpha)$ represents the class $a$. Let $\alpha_0\in CF_*(H_+)\cong CM_*(H_+)\otimes \Lambda$ be a representative of $a$ of action less than $\val(a)+\d$. Since the boundary depth (see Definition~\ref{def:boundary_depth} and the discussion following it) of $H_+$ is arbitrarily small, there exists a chain $\beta\in CF_*(H_+)$ such that $\partial_{H_+}\beta = \Phi(\alpha)-\alpha_0$, and whose action is bounded by 
        \begin{align}\label{eq:beta_action_bound}
            \cA_{H_+}(\beta)&\leq \cA_{H_+}(\Phi(\alpha)-\alpha_0) +\d \nonumber\\
            &\leq \max\{\cA_{H_+}(\Phi(\alpha)), \cA_{H_+}(\alpha_0)\}+\d\nonumber \\
            &\leq \max\{\cA_{f}(\alpha)+2\d, \val(a)+2\d\}\leq c(F;a)+3\d,
        \end{align} 
        where the third inequality follows from Lemma~\ref{lem:proj_and_cont_or_diff}.
        Let $\pi_{V^c}\beta$ be the chain obtained as the projection of $\beta$ to the subspace spanned by the critical points outside of $V$. Since $f$ and $H_+$ coincide (up to second order) on their critical points outside of $V$, we can identify  $\pi_{V^c}\beta$ with a chain in $CF_*(f)$ as well. The difference $\alpha_1:=\alpha-\partial_{f}\pi_{V^c}\beta$ is a chain representing $a$ as well.   We claim that $\alpha_1$ satisfies the assertions of the claim, namely that:
        \begin{enumerate}
            \item \label{itm:alpha_1_action_bnd} $\cA_f(\alpha_1)\leq  c(F;a)+3\d$, and 
            \item \label{itm:alpha_1_proj_action_bnd} $\cA_f(\pi_{V^c}\alpha_1)\leq \val(a)+\d$.
        \end{enumerate}
        To see \ref{itm:alpha_1_action_bnd}, it is sufficient to notice that 
        $$
        \cA_f(\partial_f\pi_{V^c}\beta) \leq \cA_f(\pi_{V^c}\beta) = \cA_{H_+}(\pi_{V^c}\beta)\leq  \cA_{H_+}(\beta)\overset{(\ref{eq:beta_action_bound})}{\leq}   c(F;a)+3\d.
        $$
        Item \ref{itm:alpha_1_proj_action_bnd} requires applying Lemmas \ref{lem:proj_and_cont_or_diff} and      \ref{lem:action_non_dec_out} to $\alpha_1$. 
        Note that, by item \ref{itm:alpha_1_action_bnd}, $\alpha_1$ satisfies the hypothesis of Lemma~\ref{lem:action_non_dec_out}, since 
        \begin{equation}\label{eq:action_bnd_alpha1}
            \cA_f(\alpha_1)\leq c(F;a)+3\d < \gw(N)+\val(a) - 4\d+3\d = \gw(N)+\val(a)-\d.
        \end{equation}
        Applying Lemma~\ref{lem:action_non_dec_out} we conclude that
        \begin{equation}\label{eq:pi_v_alpha1}
            \cA_f(\pi_{V^c}\alpha_1)\leq \max\{\cA_{H_+}(\pi_{V^c}\Phi(\alpha_1)),\val(a)\}.
        \end{equation}
        Since $\Phi$ is a chain map, we may write $\Phi(\alpha_1) = \Phi(\alpha) - \Phi\circ \partial_f \pi_{V^c}\beta = \Phi(\alpha) - \partial_{H_+}\circ \Phi(\pi_{V^c}\beta)$. Adding and subtracting $\partial_{H_+}\circ \pi_{V^c}\beta$ we can write $\Phi(\alpha_1)$ as a sum of two terms:
        \begin{align*}
        \Phi(\alpha_1) &= \Big(\Phi(\alpha) - \partial_{H_+}\circ \pi_{V^c}\beta\Big)+\Big(\partial_{H_+}\circ \pi_{V^c}\beta-\partial_{H_+}\circ \Phi(\pi_{V^c}\beta)\Big)\\
        &=\Big(\Phi(\alpha) - \partial_{H_+}\circ \pi_{V^c}\beta\Big)+\Big(\partial_{H_+}\circ (\pi_{V^c}- \Phi\pi_{V^c})\Big)\beta.
        \end{align*}
        Composing with $\pi_{V^c}$ and using the second assertion of Lemma~\ref{lem:proj_and_cont_or_diff} that $\pi_{V^c}\partial_{H_+}\pi_{V^c} = \pi_{V^c}\partial_{H_+}$, we obtain
        \begin{align*}
        \pi_{V^c}\Phi(\alpha_1) &= \Big(\pi_{V^c}\Phi(\alpha) - \pi_{V^c}\partial_{H_+}\circ \pi_{V^c}\beta\Big)+\Big(\pi_{V^c}\circ\partial_{H_+}\circ (\pi_{V^c}- \Phi\pi_{V^c})\Big)\beta\\
        &= \Big(\pi_{V^c}\Phi(\alpha) - \pi_{V^c}\partial_{H_+}\beta\Big)+\Big(\pi_{V^c}\circ\partial_{H_+}\circ \pi_{V^c}\circ (\pi_{V^c}- \Phi\pi_{V^c})\Big)\beta\\
        &= \Big(\pi_{V^c}\Phi(\alpha) - \pi_{V^c}\partial_{H_+}\beta\Big)+\Big(\pi_{V^c}\circ\partial_{H_+}\circ(\pi_{V^c}- \pi_{V^c}\Phi\pi_{V^c})\Big)\beta.
        \end{align*}
        Recalling that $\partial_{H_+}\beta = \Phi(\alpha)-\alpha_0$, this yields 
        \begin{align}\label{eq:proj_alpha1_decomp}
        \pi_{V^c}\Phi(\alpha_1) &= \Big(\pi_{V^c}\Phi(\alpha) - \pi_{V^c}(\Phi(\alpha)-\alpha_0)\Big)+\Big(\pi_{V^c}\circ\partial_{H_+}\circ(\pi_{V^c}- \pi_{V^c}\Phi\pi_{V^c})\Big)\beta \nonumber\\
        &= \pi_{V^c}\alpha_0 + \Big(\pi_{V^c}\circ\partial_{H_+}\circ(\pi_{V^c}- \pi_{V^c}\Phi\pi_{V^c})\Big)\beta
        \end{align} 
        As we chose $\alpha_0$ to have action bounded by $\val(a)+\d$, is suffices to bound the action of the second summand. By Lemma~\ref{lem:proj_and_cont_or_diff}, the map $\pi_{V^c} \circ \Phi\circ\pi_{V^c} - \pi_{V^c}$ decreases action by at least $\gw(N)-\d$. Therefore, 
        \begin{align*}
            \cA_{H_+} \Big(\pi_{V^c}\circ\partial_{H_+}\circ(\pi_{V^c}- \pi_{V^c}\Phi\pi_{V^c})\beta\Big) &\leq  \cA_{H_+}\Big((\pi_{V^c}- \pi_{V^c}\Phi\pi_{V^c})\beta\Big)\\
            &\leq \cA_{H_+}(\beta)-\gw(N)+\d\\
            &\overset{(\ref{eq:beta_action_bound})}{\leq} c(F;a)+3\d -\gw(N)+\d\\
            &< \gw(N)+\val(a)-4\d+2\d -\gw(N)+\d \\
            &= \val(a)-\d.
        \end{align*}
        Combining this with (\ref{eq:proj_alpha1_decomp}), we get
        \begin{align*}
            \cA_{H_+}(\pi_{V^c} \Phi(\alpha_1) )&\leq \max\Big\{\cA_{H_+}(\pi_{V^c}\alpha_0 ), \cA_{H_+} \Big(\pi_{V^c}\circ\partial_{H_+}\circ(\pi_{V^c}- \pi_{V^c}\Phi\pi_{V^c})\beta\Big)\Big\}\\
            &\leq \max\{\val(a)+\d,\val(a)-\d\} = \val(a)+\d.
        \end{align*}
        Together with (\ref{eq:pi_v_alpha1}) this shows  $\cA_f(\pi_{V^c}\alpha_1) \leq \val(a)+\d$ and thus concludes the proof.
        \end{proof}

\bibliographystyle{alpha}
\bibliography{refs}

\end{document}